\documentclass[a4paper,11pt]{scrartcl}
\addtokomafont{disposition}{\rmfamily}

\usepackage{amsthm}
\usepackage{amsmath}
\usepackage{amssymb}
\usepackage{mathrsfs} 
\usepackage{graphicx}
\usepackage[round]{natbib}

\usepackage[dvipsnames]{xcolor}
\usepackage{mathtools}

\usepackage[colorlinks=true,linkcolor=NavyBlue, citecolor=NavyBlue, urlcolor=NavyBlue, linktocpage=true]{hyperref} 
\usepackage{bbm}
\usepackage{textcomp}
\usepackage{enumerate}
\usepackage{mathtools}

\usepackage{color}
\usepackage{framed}
\usepackage{comment}
\definecolor{shadecolor}{gray}{0.9}

\usepackage{dsfont} 
\usepackage{caption} 
\usepackage{subcaption} 
\usepackage{graphicx}
\usepackage{pdfpages}
\usepackage{url}
\usepackage[english]{babel}
\specialcomment{extra}{\begin{shaded}}{\end{shaded}}

\theoremstyle{plain}  
\newtheorem{thm}{Theorem}[section] 
 
\newtheorem{prop}[thm]{Proposition} 
\newtheorem{cor}[thm]{Corollary} 

\theoremstyle{definition} 
\newtheorem{defn}[thm]{Definition}

\newtheorem{rem}[thm]{Remark}

\newtheoremstyle{assumption}
{3pt}
{3pt}
{}
{}
{\bf}
{.}
{.5em}
{\thmname{#1} (\thmnote{#3}\thmnumber{#2})}

\theoremstyle{assumption}
\newtheorem{ass}{Assumption}

\theoremstyle{remark} 

\newcommand{\diff}{\mathrm{d}}
\newcommand{\dint}{\,\mathrm{d}}
\newcommand{\E}{\mathbb{E}}

\newcommand{\conv}{\operatorname{conv}}
\newcommand{\ES}{\operatorname{ES}}

\newcommand{\RVaR}{\operatorname{RVaR}}

\newcommand{\eps}{\varepsilon}

\newcommand{\F}{\mathcal{F}}
\newcommand{\R}{\mathbb{R}}
\newcommand{\A}{\mathsf{A}}

\newcommand{\one}{\mathds{1}}
\newcommand{\interior}{\operatorname{int}}

\newcommand{\VaR}{\operatorname{VaR}}
\DeclareMathOperator*{\argmin}{arg\,min}

\renewcommand{\P}{\mathbb P}

\renewcommand{\a}{\alpha}
\renewcommand{\b}{\beta}
\newcommand{\g}{\gamma}

\renewcommand{\rm}{\normalfont \rmfamily}
\renewcommand{\bf}{\normalfont \bfseries}

\def\be{\begin{equation} \label}
\def\ee{\end{equation}}

\numberwithin{equation}{section} 

\usepackage[colorinlistoftodos,textsize=tiny]{todonotes}
\newcommand{\Comments}{1}
\newcommand{\mynote}[2]{\ifnum\Comments=1\textcolor{#1}{#2}\fi}
\newcommand{\mytodo}[2]{\ifnum\Comments=1%
  \todo[linecolor=#1!80!black,backgroundcolor=#1,bordercolor=#1!80!black]{#2}\fi}

\ifnum\Comments=1               
\fi

\ifnum\Comments=1               
\fi

\begin{document}

\title{Evaluating Range Value at Risk Forecasts\footnote{An earlier version of this paper was circulated under the name \textit{Elicitability of Range Value at Risk}.}}
\author{Tobias Fissler\thanks{WU Vienna University of Economics and Business, Department of Finance, Accounting and Statistics, Welthandelsplatz 1, 1020 Vienna, Austria, 
		e-mail: \href{mailto:tobias.fissler@wu.ac.at}{tobias.fissler@wu.ac.at}} 
		\and Johanna F.~Ziegel\thanks{University of Bern, Department of Mathematics and Statistics, Institute of Mathematical Statistics and Actuarial Science, Alpeneggstrasse 22, 3012 Bern, Switzerland, 
		e-mail: \href{mailto:johanna.ziegel@stat.unibe.ch}{johanna.ziegel@stat.unibe.ch}}}

\maketitle

\begin{abstract}
\textbf{Abstract.}
The debate of what quantitative risk measure to choose in practice has mainly focused on the dichotomy between Value at Risk (VaR)\,---\,a quantile\,---\,and Expected Shortfall (ES)\,---\,a tail expectation. Range Value at Risk (RVaR) is a natural interpolation between these two prominent risk measures, which constitutes a tradeoff between the sensitivity of the latter and the robustness of the former, turning it into a practically relevant risk measure on its own. 
As such, there is a need to statistically validate RVaR forecasts and to compare and rank the performance of different RVaR models, tasks subsumed under the term `backtesting' in finance.
The predictive performance is best evaluated and compared in terms of strictly consistent loss or scoring functions. That is, functions which are minimised in expectation by the correct RVaR forecast.
Much like ES, it has been shown recently that RVaR does not admit strictly consistent scoring functions, i.e., it is not elicitable.
Mitigating this negative result, this paper shows that a triplet of RVaR with two VaR components at different levels is elicitable. We characterise the class of strictly consistent scoring functions for this triplet. Additional properties of these scoring functions are examined, including the diagnostic tool of Murphy diagrams. The results are illustrated with a simulation study, and we put our approach in perspective with respect to the classical approach of trimmed least squares in robust regression.
\end{abstract}

\noindent
\textit{MSC2020 classes:} 62C99; 62G35; 62P05; 91G70

\section{Introduction}


In the field of quantitative risk management, the last one or two decades have seen a lively debate about which monetary risk measure \citep{ArtznerDelbaenETAL1999} be best in (regulatory) practice. The debate mainly focused on the dichotomy between Value at Risk ($\VaR_\b$) on the one hand and Expected Shortfall ($\ES_\b$) on the other hand, at some probability level $\b\in(0,1)$ (see Section \ref{sec:notation} for definitions). 
Mirroring the historical joust between median and mean as centrality measures in classical statistics, $\VaR_\b$, basically a quantile, is esteemed for its robustness, while $\ES_\b$, a tail expectation, is deemed attractive due to its sensitivity and the fact that it satisfies the axioms of a coherent risk measure \citep{ArtznerDelbaenETAL1999}.
We refer the reader to \cite{EmbrechtsETAL2014} and \cite{EmmerKratzTasche2015} for comprehensive academic discussions, and to \cite{BIS2014} for a regulatory perspective in banking.

\cite{ContDeguestETAL2010} considered the issue of statistical robustness of risk measure estimates in the sense of \cite{Hampel1971}. They showed that a risk measure cannot be both robust and coherent. As a compromise, they propose the risk measure `Range Value at Risk', $\RVaR_{\a,\b}$ at probability levels $0<\a<\b<1$.
It is defined as the average of all $\VaR_\gamma$ with $\gamma$ between $\a$ and $\b$ (see Section \ref{sec:notation} for definitions).
As limiting cases, one obtains $\RVaR_{\b,\b} = \VaR_\b$ and $\RVaR_{0,\b} = \ES_\b$, which presents $\RVaR_{\a,\b}$ as a natural interpolation of $\VaR_\b$ and $\ES_\b$. Quantifying its robustness in terms of the \emph{breakdown point} 
and following the arguments provided in \citet[p.\ 59]{HuberRonchetti2009}, $\RVaR_{\a,\b}$ has a breakdown point of $\min\{\a,1-\b\}$, placing it between the very robust $\VaR_\b$ (with a breakdown point of $\min\{\b,1-\b\}$) and the entirely non-robust $\ES_\b$ (breakdown point 0). This means it is a robust\,---\,and hence, not coherent\,---\,risk measure, unless it degenerates to $\RVaR_{0,\b} = \ES_\b$ (or if $0\le \a<\b=1$).
Moreover, $\RVaR$ belongs to the wide class of distortion risk measures \citep{Kusuoka2001}. For further contributions to robustness in the context of risk measures, we refer the reader to \cite{KraetschmeSchiedZaehle2012, KraetschmeSchiedZaehle2014}, \cite{KouPengHeyde2013}, \cite{EmbrechtsWangWang2015} and \cite{Zaehle2016}.
Since the influential article \cite{ContDeguestETAL2010}, RVaR has gained increasing attention in the risk management literature\,---\,see \cite{EmbrechtsETAL2018, EmbrechtsLiuWang2018} for extensive studies\,---\,as well as in econometrics \citep{Barendse2020} where RVaR sometimes has the alternative denomination \emph{Interquantile Expectation}. 
For the symmetric case $\b = 1-\a>1/2$, $\RVaR_{\a,1-\a}$ is known under the term $\alpha$-trimmed mean in classical statistics and it constitutes an alternative to and interpolation of the mean and the median as centrality measures; see \cite{LM2019} for a recent study and a multivariate extension of the trimmed mean. It is closely connected to the $\a$-Winsorized mean, see \eqref{eq:W}.


How to evaluate
the predictive performance of point forecasts, $x_t$, for a statistical functional $T$, such as the mean, median or a risk measure, of the (conditional) distribution of a quantity of interest, $y_t$?
It is commonly measured in terms of the \emph{average realised score} $\frac{1}{n}\sum_{t=1}^n S(x_t,y_t)$ for some scoring or loss function $S$, using the orientation the smaller the better. 
Consequently, the loss function $S$ should be \emph{strictly consistent} for $T$ in that $T(F) = \argmin_x \int S(x,y)\dint F(y)$: Correct predictions are honoured and encouraged in the long run.
E.g., the squared loss $S(x,y) = (x-y)^2$ is consistent for the mean, and the absolute loss $S(x,y) = |x-y|$ is consistent for the median.
If a functional admits a strictly consistent score, it is called \emph{elicitable} \citep{Osband1985, LambertETAL2008, Gneiting2011}. By definition, elicitable functionals allow for $M$-estimation and have natural estimation paradigms in regression frameworks \cite[Section 2]{DFZ2020}, such as quantile regression \citep{KoenkerBasset1978, Koenker2005} or expectile regression \citep{NeweyPowell1987}.
Elicitability is crucial for meaningful forecast evaluation \citep{EngelbergManskiETAL2009, MurphyDaan1985, Gneiting2011}.
In the context of probabilistic forecasts with distributional forecasts $F_t$ or density forecasts $f_t$, (strictly) consistent scoring functions are often referred to as \emph{(strictly) proper rules}, such as the log-score $S(f,y) = -\log f(y)$ \citep{GneitingRaftery2007}.
In quantitative finance, and particularly in the debate about which risk measure is best in practice, elicitability has gained considerable attention \citep{EmmerKratzTasche2015, Ziegel2016, Davis2016}. Especially, the role of elicitability for backtesting purposes has been highly debated \citep{Gneiting2011, AcerbiSzekely2014, AcerbiSzekely2017}. It has been clarified that elicitability is central for \emph{comparative} backtesting \citep{FisslerETAL2016, NoldeZiegel2017}. 

Not all functionals are elicitable. \cite{Osband1985} showed that an elicitable functional necessarily has convex level sets (CxLS): If $T(F_0) = T(F_1)=t$ for two distributions $F_0,F_1$, then $T(F_\lambda)=t$ where $F_\lambda =(1-\lambda)F_0 + \lambda F_1$, $\lambda\in(0,1)$.
Variance and ES generally do not have CxLS \citep{Weber2006, Gneiting2011}, therefore failing to be elicitable. 
The \emph{revelation principle} \citep{Osband1985, Gneiting2011} asserts that any bijection of an elicitable functional is elicitable. This implies that the pair (mean, variance)\,---\,being a bijection of the first two moments\,---\,is elicitable despite the variance failing to be elicitable. Similarly, \cite{FisslerZiegel2016} showed that the pair $(\VaR_\b, \ES_\b)$ is elicitable with the structural difference that the revelation principle is not applicable in this instance. This gave rise to the more general finding that the minimal expected score and its minimiser are always jointly elicitable \citep{Brehmer2017, FrongilloKash2020}.


Recently, \citet[Theorem 5.3]{WangWei2020} showed that $\RVaR_{\a,\b}$, $0<\a<\b<1$, similarly to $\ES_\a$, fails to have the CxLS property, which rules out its elicitability. In contrast, they observe that the identity 
\be{eq:ES difference}
\RVaR_{\a,\b} = \big(\b \ES_\b - \a\ES_\a\big)/(\b-\a), \qquad0<\a<\b<1,
\ee 
and the CxLS property of the pair $(\VaR_\a, \ES_\a)$ implies the CxLS property of the triplet $(\VaR_\a, \VaR_\b, \RVaR_{\a,\b})$ \cite[Example 4.6]{WangWei2020}, leading to the question whether this triplet is elicitable or not.
Invoking the elicitability of $(\VaR_\a, \ES_\a)$, the identity at \eqref{eq:ES difference} and the revelation principle establishes the elicitability of the quadruples $(\VaR_\a, \VaR_\b, \ES_\a, \RVaR_{\a,\b})$ and $(\VaR_\a, \VaR_\b, \ES_\b, \RVaR_{\a,\b})$. This approach has already been used in the context of regression in \cite{Barendse2020}.

\emph{A fortiori}, we show that the triplet $(\VaR_\a, \VaR_\b, \RVaR_{\a,\b})$ is elicitable (Theorem \ref{thm:sufficiency}) under weak regularity conditions. Practically, opens the way to meaningful forecast performance comparison, and in particular comparative backtests, of this triplet, as well as to a regression framework.
Theoretically, this shows that the elicitation complexity \citep{LambertETAL2008, FrongilloKash2020} or elicitation order \citep{FisslerZiegel2016} of $\RVaR_{\a,\b}$ ist at most 3.
Moreover, requiring only VaR-forecasts besides the RVaR-forecast is particularly advantageous to additionally requiring an ES-forecasts since 
the triplet $(\VaR_\a(F), \VaR_\b(F), \RVaR_{\a,\b}(F))$, $0<\a<\b<1$, exists and is finite for any distribution $F$, whereas 
$\ES_\a(F)$ and $\ES_\b(F)$ only exist if the (left) tail of the distribution $F$ is integrable. As $\RVaR_{\a,\b}$ is used often for robustness purposes, safeguarding against outliers and heavy-tailedness, this advantage is important.

We would like to point out the structural difference between the elicitability result of $(\VaR_\a, \VaR_\b, \RVaR_{\a,\b})$ provided in this paper and the one concerning $(\VaR_\a, \ES_\a)$ in \cite{FisslerZiegel2016} as well as the more general results of \cite{FrongilloKash2020} and \cite{Brehmer2017}. While $\ES_\a$ corresponds to the negative of a minimum of an expected score which is strictly consistent for $\VaR_\a$, it turns out that $\RVaR_{\a,\b}$ can be represented as the \emph{difference} of minima of expected strictly consistent scoring functions for $\VaR_\a$ and $\VaR_\b$ (Proposition \ref{prop:identifiability}). As a consequence, the class of strictly consistent scoring functions for the triplet $(\VaR_\a, \VaR_\b, \RVaR_{\a,\b})$ turns out to be less flexible than the one for $(\VaR_\a, \ES_\a)$; see Remark \ref{rem:difference of minima} for details. 
In particular, there is essentially no translation invariant or positively homogeneous scoring function which is strictly consistent for $(\VaR_\a, \VaR_\b, \RVaR_{\a,\b})$; see Section \ref{sec:secondarycrit}.

The paper is organised as follows. In Section \ref{sec:notation}, we introduce the relevant notation and definitions concerning RVaR, scoring functions and elicitability. The main results establishing the elicitability of the triplet $(\VaR_\a, \VaR_\b, \RVaR_{\a,\b})$ (Theorems \ref{thm:sufficiency} and \ref{thm:necessity}) and related findings are presented in Section \ref{sec:results}. Section \ref{sec:secondarycrit} shows that there are basically no strictly consistent scoring functions for $(\VaR_\a, \VaR_\b, \RVaR_{\a,\b})$ which are positively homogeneous or translation invariant. In Section \ref{sec:mixture}, we establish a mixture representation of the strictly consistent scoring functions in the spirit of \cite{EhmETAL2016}. This result allows to compare forecasts simultaneously with respect to \emph{all} consistent scoring functions in terms of \emph{Murphy diagrams}. We demonstrate the applicability of our results and compare the discrimination ability of different scoring functions in a simulation study presented in Section \ref{sec:sim}. The paper finishes in Section \ref{sec:discussion} with a discussion of our results in the context of $M$-estimation and compares them to other suggestions in the statistical literature, in variants of a \emph{trimmed least squares} procedure \citep{KoenkerBasset1978, RuppertCarroll1980, Rousseeuw1984}.

\section{Notation and Definitions}\label{sec:notation}

\subsection{Definition of Range Value at Risk}

There are different sign conventions in the literature on risk measures. In this paper we use the following convention: If a random variable $Y$ models the losses and gains, then positive values of $Y$ represent gains and negative values of $Y$ losses. 
Moreover, if $\rho$ is a risk measure, we assume that $\rho(Y)\in\R$ corresponds to the maximal amount of money one can \emph{withdraw} such that the position $Y-\rho(Y)$ is still acceptable. Hence, negative values of $\rho$ correspond to risky positions. 
In the sequel, let $\F_0$ be the class of probability distribution functions on $\R$. Recall that the $\a$-quantile, $\a\in[0,1]$ of $F\in\F_0$ is defined as the set $q_\a(F) = \{x\in\R\,|\,F(x-)\le \a \le F(x)\}$, where $F(x-):= \lim_{t\uparrow x}F(t)$.

\begin{defn}
\emph{Value at Risk} of $F\in \F_0$ at level $\a\in[0,1]$ is defined as 
$\VaR_\a(F) = \inf q_\a(F)$. 
\end{defn}
For any $\alpha\in[0,1]$ we introduce the following subclasses of $\F_0$:
\begin{align}
\F^{\a}&= \big\{F\in\F_0\,|\,q_\a(F) = \{\VaR_\a(F)\}\big\},
&\F^{(\a)}= \big\{F\in\F_0\,|\,F(\VaR_\a(F))=\a\big\}.
\end{align}
\begin{defn}
\label{defn:IQE}
\emph{Range Value at Risk} of $F\in \F_0$ at levels $0\le \a\le \b\le 1$ is defined as 
\[
\RVaR_{\a,\b}(F) =
\begin{dcases}
\frac{1}{\b-\a}\int_\a^\b \VaR_\g(F)\dint \g, & \text{if } \a < \b,\\[0.5em]
\VaR_\a(F), & \text{if } \a=\b.
\end{dcases}
\]
\end{defn}
The definition of RVaR implies that
\be{eq:inequality}
\VaR_{\a}(F)\le \RVaR_{\a, \b}(F) \le \VaR_{\b}(F).
\ee
For $0<\a\le\b<1$ and $F\in\F_0$ one obtains that 
(i) $\RVaR_{\a,\b}(F)\in\R$;
(ii) $\RVaR_{0,\b}(F)\in\R\cup\{-\infty\}$ and it is finite if and only if $\int_{-\infty}^0|y|\dint F(y)<\infty$; and
(iii) $\RVaR_{\a,1}(F)\in\R\cup\{\infty\}$ and it is finite if and only if $\int_{0}^\infty|y|\dint F(y)<\infty$.
$\RVaR_{0,1}(F)$ exists only if $\int_{-\infty}^0|y|\dint F(y)<\infty$ or $\int_{0}^\infty|y|\dint F(y)<\infty$. If $F$ has a finite first moment, then $\RVaR_{0,1}(F) = \int y \dint F(y)$ coincides with the first moment of $F$.
Provided that $\RVaR_{\a,\b}(F)$ exists it holds that
\be{eq:correction}
\begin{aligned}
\RVaR_{\a,\b}(F) &= \frac{1}{\b-\a}\Bigg(\int_{(\VaR_\a(F),\VaR_\b(F)]}y \dint F(y) \\ &+\VaR_\a(F)\big(F(\VaR_\a(F)) - \a \big) - \VaR_\b(F)\big(F(\VaR_\b(F)) - \b\big) \Bigg),
\end{aligned}
\ee
using the usual conventions $F(-\infty) =0$, $F(\infty)=1$ and $0\cdot\infty = 0\cdot (-\infty) = 0$.
If $F\in\F^{(\a)}\cap\F^{(\b)}$ then the correction terms in the second line of \eqref{eq:correction} vanish, yielding
$\RVaR_{\a,\b}(F) = \E_F[Y\,\one\{\VaR_\a(F)< Y\le \VaR_\b(F)\}]/(\b-\a)$,
which justifies an alternative name for RVaR, namely \emph{Interquantile Expectation}.

\begin{defn}
\emph{Expected Shortfall} of $F\in\F_0$ at level $\a\in(0,1)$ is defined as 
\(
\ES_\a(F) =\RVaR_{0,\a}(F)\in \R\cup\{-\infty\}.
\)
\end{defn}
Hence, provided that $\ES_\a(F), \ES_\b(F)$ are finite, one obtains the identity \eqref{eq:ES difference}.
If $F$ has a finite left tail ($\int_{-\infty}^0 |y|\dint F(y)<\infty$) then one could use the right hand side of \eqref{eq:ES difference} as a definition of $\RVaR_{\a,\b}(F)$. However, in line with our discussion in the introduction, $\RVaR_{\a,\b}(F)$ always exists and is finite for $0<\a<\b<1$ even if the right hand side of \eqref{eq:ES difference} is not defined. 

Interestingly, \citet[Theorem 2]{EmbrechtsLiuWang2018} establish that $\RVaR$ can be written as an inf-convolution of $\VaR$ and $\ES$ at appropriate levels. 
This result amounts to a sup-convolution in our sign convention. Also note that our parametrisation of of $\RVaR_{\a,\b}$ differs from theirs.

For $\a\in(0,1/2)$, $\RVaR_{\a, 1-\a}$ corresponds to the \emph{$\a$-trimmed mean} and has a close connection to the \emph{$\a$-Winsorized mean} $W_\a$ \cite[pp.\ 57--59]{HuberRonchetti2009} via
\be{eq:W}
W_\a(F):= (1-2\a) \RVaR_{\a, 1-\a}(F) + \a \VaR_\a(F) + \a \VaR_{1-\a}(F), \quad \a\in(0,1/2).
\ee

\subsection{Elicitability and scoring functions}

Using the decision-theoretic framework of \cite{FisslerZiegel2016} and \cite{Gneiting2011}, we introduce the following notation.
Let $\F\subseteq \F_0$ be some generic subclass, and $\A\subseteq\R^k$ be an \emph{action domain}.
Whenever we consider a functional $T\colon\F\to\A$, we tacitly assume that $T(F)$ is well-defined for all $F\in\F$ and is an element of $\A$. $T(\F)$ corresponds to the image $\{T(F)\in\A\,|\,F\in\F\}$. For any subset $M\subseteq \R^k$ we denote with $\interior(M)$ the largest open subset of $M$. Moreover, $\conv(M)$ denotes the convex hull of the set $M$.

We say that a function $a\colon\R\to\R$ is $\F$-integrable if it is measurable and $\int |a(y)|\dint F(y)<\infty$ for all $F\in\F$. Similarly, a function $g\colon\A\times\R\to\R$ is called $\F$-integrable if $g(x,\cdot)\colon \R\to\R$ is $\F$-integrable for all $x\in\A$. If $g$ is $\F$-integrable, we define the map
$\bar g\colon \A\times \F\to\R$, $\bar g(x,F) := \int g(x,y)\dint F(y)$.
If $g\colon\A\times\R\to\R$ is sufficiently smooth in its first argument, we denote the $m$th partial derivative of $g(\cdot, y)$ with $\partial_m g(\cdot, y)$.

\begin{defn}
A map $S\colon\A\times\R\to\R$ is an \emph{$\F$-consistent scoring function} for $T\colon\F\to\A$ if it is 
$\F$-integrable and if $\bar S(T(F),F)\le \bar S(x,F)$ for all $x\in\A$, $F\in\F$. It is \emph{strictly} $\F$-consistent for $T$ if it is consistent and if $\bar S(T(F),F) =  \bar S(x,F)$ implies that $x=T(F)$ for all $x\in\A$ and for all $F\in\F$. 
A functional $T\colon\F\to\A$ is \emph{elicitable} on $\F$ if it possesses a strictly $\F$-consistent scoring function.
\end{defn}

\begin{defn}
Two scoring function $S, \widetilde S\colon\A\times\R\to\R$ are \emph{equivalent} if there is some $a\colon\R\to\R$ and some $\lambda>0$ such that $\widetilde S(x,y) = \lambda S(x,y) + a(y)$ for all $(x,y)\in\A\times\R$.
They are \emph{strongly} equivalent if additionally $a\equiv0$.
\end{defn}
This equivalence relation preserves (strict) consistency: If $S$ is (strictly) $\F$-consistent for $T$ and if $a$ is $\F$-integrable, then $\widetilde S$ is also (strictly) $\F$-consistent for $T$.
Closely related to the concept of elicitability is the notion of \emph{identifiability}.
\begin{defn}
A map $V\colon\A\times \R\to\R^k$ is an \emph{$\F$-identification function} for $T\colon\F\to\A$ if it is 
$\F$-integrable and if $\bar V(T(F),F) = 0$ for all $F\in\F$. It is a \emph{strict} $\F$-identification function for $T$ if additionally $\bar V(x,F)=0$ implies that $x=T(F)$ for all $x\in\A$ and for all $F\in\F$.
it is consistent and if $\bar S(T(F),F) =  \bar S(x,F)$ implies that $x=T(F)$ for all $x\in\A$ and for all $F\in\F$. 
A functional $T\colon\F\to\A$ is \emph{elicitable} if it possesses a strictly $\F$-consistent scoring function.
A functional $T\colon\F\to\A$ is \emph{identifiable} on $\F$ if it possesses a strict $\F$-identification function.
\end{defn}
In contrast to \cite{Gneiting2011} we consider point-valued functionals only. For a recent comprehensive study on elicitability of set-valued functionals we refer to \cite{FisslerFrongilloHlavinovaRudloff2020}.
For the sake of completeness, we list some assumptions used in Section \ref{sec:results} which were originally introduced in \cite{FisslerZiegel2016} in the Appendix.

\section{Elicitability and identifiability results}\label{sec:results}

\citet[Theorem 5.3]{WangWei2020} show that for $0<\a<\b<1$, $
\RVaR_{\a,\b}$ (and also the pairs $(\VaR_{\a}, \RVaR_{\a,\b})$ and $(\VaR_{\b}, \RVaR_{\a,\b})$) do not have CxLS on $\F_{\text{dis}}$, the class of distributions with bounded and discrete support. Hence, invoking that CxLS are necessary for elicitability and identifiability, $
\RVaR_{\a,\b}$ and the pairs $(\VaR_{\a}, \RVaR_{\a,\b})$ and $(\VaR_{\b}, \RVaR_{\a,\b})$ fail to be elicitable and identifiable on $\F_{\text{dis}}$.
Our novel contribution is that the \emph{triplet} $(\VaR_\a, \VaR_\b, \RVaR_{\a, \b})$, however, is elicitable and identifiable, subject to mild conditions.
We use the notation $S_\a(x,y) = (\one\{y\le x\} - \a)x - \one\{y\le x\}y$, and recall that $S_\a$ is $\F$-consistent for $\VaR_\a$ 
if $\int_{-\infty}^0 |y|\dint F(y)<\infty$ for all $F\in\F$, 
and strictly $\F$-consistent if furthermore $\F\subseteq \F^{\a}$ \citep{Gneiting2011}.

\begin{prop}\label{prop:identifiability}
For $0<\a<\b<1$ the map $V\colon\R^3\times \R \to\R^3$ 
\begin{align}\label{eq:identification}
V(x_1,x_2,x_3,y) \
= \begin{pmatrix}
\one\{y\le x_1\} -\a \\
\one\{y\le x_2\} -\b \\
x_3 + \frac{1}{\b - \a} 
\big(S_\b(x_2,y) - S_\a(x_1,y)\big)
\end{pmatrix}
\end{align}
is an $\F^{(\a)}\cap\F^{(\b)}$-identification function for $(\VaR_\a, \VaR_\b, \RVaR_{\a, \b})$, which is strict on $\F^{\a}\cap\F^{(\a)}\cap\F^{\b}\cap\F^{(\b)}$.
\end{prop}
\begin{proof}
The proof is standard, observing that 
\be{eq:V_3}
\bar V_3(\VaR_\a(F),\VaR_\b(F),x_3,F) = x_3 - \RVaR_{\a,\b}(F),
\ee
which follows from the representation \eqref{eq:correction}.
\end{proof}

The following theorem establishes a rich class of (strictly) consistent scoring functions $S\colon\R^3\times\R\to\R$ for $(\VaR_{\a}, \VaR_{\b}, \RVaR_{\a, \b})$. By \textit{a priori} assuming forecasts to be bounded with values in some cube $[c_{\min}, c_{\max}]^3$, $-\infty\le c_{\min}<c_{\max}\le \infty$ (with the tacit convention that $[c_{\min}, c_{\max}] := [c_{\min}, c_{\max}]\cap  \R$ if $c_{\min}=-\infty$ or $c_{\max}=\infty$), the class gets even broader.

\begin{thm}\label{thm:sufficiency}
For $0<\a<\b<1$, the map $S\colon[c_{\min}, c_{\max}]^3 \times\R\to \R$\begin{align}\label{eq:S}
S(x_1,x_2,x_3,y)
&=  \big(\one\{y\le x_1\} - \a\big)g_1(x_1) - \one\{y\le x_1\}g_1(y) \\ \nonumber
&+\big(\one\{y\le x_2\} - \b\big)g_2(x_2) - \one\{y\le x_2\}g_2(y)\\ \nonumber
&+ \phi'(x_3)\Big(x_3 + \frac{1}{\b - \a} \big(S_\b(x_2,y) - S_\a(x_1,y) \big)\Big)- \phi(x_3) +a(y),
\end{align}
is an $\F$-consistent scoring function for $(\VaR_{\a}, \VaR_{\b}, \RVaR_{\a, \b})$ 
if 
\begin{enumerate}[\rm (i)]
\item
\label{item:conv}
$\phi\colon [c_{\min}, c_{\max}]\to\R$ is convex with subgradient $\phi'$,
\item
\label{item:incr}
for all $x_3\in[c_{\min}, c_{\max}]$ the functions
\begin{align}
\label{eq:G1}
&G_{1,x_3}\colon [c_{\min}, c_{\max}]\to \R, \qquad x_1\mapsto g_1(x_1) - x_1 \phi'(x_3) /(\b-\a), \\
\label{eq:G2}
&G_{2,x_3}\colon [c_{\min}, c_{\max}] \to \R, \qquad x_2\mapsto g_2(x_2) + x_2 \phi'(x_3) /(\b-\a)
\end{align}
are increasing, and
\item
\label{item:integr}
$y\mapsto a(y) - \one\{y\le x_1\}g_1(y)- \one\{y\le x_2\}g_2(y)$ is $\F$-integrable for all $x_1,x_2\in [c_{\min}, c_{\max}]$.
\end{enumerate}
If moreover $\phi$ is strictly convex, and the functions at \eqref{eq:G1} and \eqref{eq:G2} are strictly increasing, then $S$ is strictly $\F^\a\cap\F^\b$-consistent for $T$.
\end{thm}

\begin{proof}
Let $(x_1,x_2,x_3)\in\A$, $F\in\F$ and $(t_1,t_2,t_3):=T(F)$. Then, since $G_{1,x_3}$ is increasing, $[c_{\min},c_{\max}]\times \R\ni (x_1',y) \mapsto S(x_1',x_2,x_3,y)$ is $\F$-consistent for $\VaR_\a$ and it is strictly $\F^\a$-consistent if $G_{1,x_3}$ is strictly increasing. Similar comments apply to the map $[c_{\min},c_{\max}]\times \R\ni (x_2',y) \mapsto S(t_1,x_2',x_3,y)$. Hence,
\begin{align*}
0 &\le \bar S(x_1,x_2,x_3,F) - \bar S(t_1,x_2,x_3,F) + \bar S(t_1,x_2,x_3,F) - \bar S(t_1,t_2,x_3,F)\\
&= \bar S(x_1,x_2,x_3,F) - \bar S(t_1,t_2,x_3,F)
\end{align*}
with a strict inequality under the conditions for strict consistency and if $(x_1,x_2)\neq (t_1,t_2)$.
Finally,
\begin{equation}\label{eq:proof_3}
\bar S(t_1,t_2,x_3,F) - \bar S(t_1,t_2,t_3,F) 
= \phi'(x_3)(x_3 - t_3) - \phi(x_3) + \phi(t_3)\ge0,
\end{equation}
since $\phi$ is convex. If $\phi$ is strictly convex and if $x_3\neq t_3$, the inequality in \eqref{eq:proof_3} is strict.
\end{proof}

\begin{rem}\label{rem:thm suff}
Provided condition \eqref{item:integr} in Theorem \ref{thm:sufficiency} holds and 
if $\phi$ is strictly convex, and $G_{1,x_3}$ and $G_{2,x_3}$ strictly increasing then $S$ given in \eqref{eq:S} 
is still strictly $\F$-consistent in the $\RVaR$-component for general $\F\subseteq \F_0$. That is, for $F\in\F$
\[
\argmin_{x\in\A_0} \bar S(x,F) = q_\a(F)\times q_\b(F) \times \{\RVaR_{\a,\b}(F)\}.
\]
\end{rem}

Making use of \eqref{eq:W} and the revelation principle \citep{Osband1985, Gneiting2011, Fissler2017}, Theorem \ref{thm:sufficiency} also provides a rich class of strictly consistent scoring function for $(\VaR_\a, \VaR_{1-\a}, W_\a)$, where $W_\a$ is the $\a$-Winsorized mean.
The following proposition is useful to construct examples; see Section \ref{sec:sim}. 

\begin{prop}\label{prop:simplecond}
Let $S$ be of the form \eqref{eq:S} with a (strictly) convex and non-constant function $\phi$, and functions $g_1$, $g_2$ such that the functions at \eqref{eq:G1} and \eqref{eq:G2} are (strictly) increasing and condition \eqref{item:integr} of Theorem \ref{thm:sufficiency} is satisfied. 
Then the following holds:
\begin{enumerate}[\rm (i)]
\item The subgradient $\phi'$ of $\phi$ is necessarily bounded and the one-sided derivatives of $g_1$ and $g_2$ are necessarily bounded from below.
\item $S$ is strongly equivalent to a scoring function $\tilde{S}$ of the form $\eqref{eq:S}$ with a (strictly) convex function $\tilde\phi$ such that $\tilde \phi'$ is bounded with $\beta-\alpha = -\inf_{x \in [c_{\min},c_{\max}]} \tilde\phi'(x)=\sup_{x \in [c_{\min},c_{\max}]} \tilde\phi'(x) $, and strictly increasing functions $\tilde g_1$, $\tilde g_2$ such that their one-sided derivatives are bounded from below by one and such that such that the functions at \eqref{eq:G1} and \eqref{eq:G2} are (strictly) increasing and condition \eqref{item:integr} of Theorem \ref{thm:sufficiency} is satisfied.
\end{enumerate}
\end{prop}
\begin{proof}
\begin{enumerate}[\rm (i)]
\item
The proof is similar to the one of Corollary 5.5 in \cite{FisslerZiegel2016}: 
Condition \eqref{item:incr} implies that for any  $x_1, x_1', x_2,x_2', x_3\in[c_{\min},c_{\max}]$ with $x_1<x_1'$ and $x_2<x_2'$ it holds that 
 \be{eq:bounded phi}
 -\infty< -\frac{g_2(x'_2) - g_1(x_2)}{x_2' - x_2}\le \frac{\phi'(x_3)}{\b-\a} \le \frac{g_1(x_1') - g_1(x_1)}{x_1' - x_1}<\infty.
 \ee 
 Therefore, $\phi'$ is bounded, and the one-sided derivative of $g_1$ is bounded from below by $\sup_{x_3} \phi'(x_3)/(\b-\a)$ while the one-sided derivative of $g_2$ is bounded from below by $-\inf_{x_3}\phi'(x_3)/(\b-\a)$.
\item 
For any $c \in \R$, if we replace $\phi$ with $\widehat \phi: x \mapsto \phi(x) + c x $, $g_1$ with $\widehat g_1: x \mapsto g_1(x) + cx/(\beta-\alpha)$, and $g_2$ with $\widehat g_2: x \mapsto g_2(x) - cx/(\beta-\alpha)$ in the formula \eqref{eq:S} for $S$, then $S$ does not change. Also $\widehat \phi$ is (strictly) convex if and only if $\phi$ is (strictly) convex. Furthermore, conditions \eqref{item:incr} and \eqref{item:integr} of Theorem \ref{thm:sufficiency} hold for $\phi$, $g_1$, $g_2$ if and only if they hold for $\widehat \phi$, $\widehat g_1$ and $\widehat g_2$. By part (i) of the proposition $\phi'$ is bounded. Therefore, we can assume without loss of generality that $-\inf_{x \in [c_{\min},c_{\max}]} \phi'(x)=\sup_{x \in [c_{\min},c_{\max}
} \phi'(x)=\lambda>0$, since $\phi$ is non-constant. 
Then the argument follows by setting $\tilde S = \frac{\lambda}{\b-\a}S$.
\end{enumerate}
\end{proof}

Invoking the inequality \eqref{eq:inequality} the triplet $(\VaR_{\a}, \VaR_{\b}, \RVaR_{\a, \b})$ can only attain values in the domain $\A_0:= \{(x_1,x_2,x_3)\in\R^3\,|\, x_1\le x_3 \le x_2\}$. 
Therefore, we call $\A_0$ the \emph{maximal sensible action domain}. 
Issuing forecasts for $T$ outside $\A_0$, thus violating \eqref{eq:inequality} would be irrational, corresponding to, say, negative variance forecasts. 
Still, the scoring functions of the form \eqref{eq:S} allow for the evaluation of forecasts violating \eqref{eq:inequality}.
Striving for a necessary characterisation result of (strictly) consistent scoring functions for $(\VaR_{\a}, \VaR_{\b}, \RVaR_{\a, \b})$, it is immediate to realise that there is flexibility in $[c_{\min},c_{\max}]^3 \setminus \A_0$ since one could possibly set the score to infinity there and would still preserve (strict) consistency.
Therefore, it is not astonishing that a necessary characterisation result works only on domains $\A\subseteq \A_0$.
The key to such a necessary characterisation is 
Osband's principle \cite[Theorem 3.2]{FisslerZiegel2016} originating from the seminal dissertation of \cite{Osband1985}.
Since it exploits a first-order condition of the minimisation of the expected score, the main assumptions of the result consist of smoothness assumptions on expected score as well as richness assumptions on the underlying class of distributions $\F$; see Appendix for the detailed technical formulations and \cite{FisslerZiegel2016} for a discussion of these conditions.

We introduce the class $\F_{\mathrm{cont}}\subset \F_0$ of distributions which are continuously differentiable and with a strictly positive derivative\,/\,density. (Clearly $\F_{\mathrm{cont}} \subset \F^{\gamma} \cap \F^{(\gamma)}$ for any $\gamma\in(0,1)$.)
For any $\A\subseteq \R^3$, we denote the projections on the $r$th component by $\A'_r:= \{x_r\in\R\,|\,\exists (z_1,z_2,z_3)\in\A,\ z_r = x_r\}$, $r\in\{1,2,3\}$.
For any $x_3\in\A'_3$ and $m\in\{1,2\}$, let
$\A'_{m,x_3}:= \{x_m\in\R\,|\,\exists (z_1,z_2, z_3)\in\A,\, z_m = x_m, \, z_3 = x_3\}$.

\begin{thm}\label{thm:necessity}
Let $\F\subseteq \F_{\mathrm{cont}}$, $0<\a<\b<1$, $T=(\VaR_{\a}, \VaR_{\b}, \RVaR_{\a, \b})\colon \F\to\A\subseteq \A_0$, and let $V=(V_1,V_2,V_3)^\intercal$ defined at \eqref{eq:identification}.
If  Assumptions (V\ref{ass:V1}), and (F\ref{ass:F1}) hold and $(V_1,V_2)^\intercal$ satisfies Assumption (V\ref{ass:V4}), then any 
strictly $\F$-consistent scoring function $S\colon \A\times \R\to\R$ for $T$ that satisfies assumptions (VS\ref{ass:VS1}) and (S\ref{ass:S2}) is necessarily of the form \eqref{eq:S} almost everywhere, where the functions $G_{r,x_3}\colon\A'_{r,x_3}\to\R$, $r\in\{1,2\}$, $x_3\in\A'_3$, in \eqref{eq:G1} and \eqref{eq:G2} are strictly increasing and $\phi\colon \A'_3\to\R$ is strictly convex.
\end{thm}


\begin{proof}
First note that $V$ satisfies assumption (V\ref{ass:V3}) on $\F\subseteq \F_{\mathrm{cont}}$.
Let $F\in\F$ with derivative $f$ and let $x\in\interior(\A)$. Then one obtains
\[
\bar V_3(x,F) = x_3  + \frac{1}{\b - \a} \left(x_2(F(x_2) - \b) - x_1(F(x_1) - \a) - \int_{x_1}^{x_2} yf(y)\dint y\right)
\]
The partial derivatives of $V$ are given by
$\partial_1 \bar V_1(x,F) = f(x_1)$, 
$\partial_2 \bar V_2(x,F) = f(x_2)$,
$\partial_1 \bar V_3(x,F) =-(F(x_1) - \a)/(\b-\a)$,
$\partial_2 \bar V_3(x,F) =(F(x_2) - \b)/(\b-\a)$,
$\partial_3 \bar V_3(x,F) =1$,
and $\partial_r \bar V_1(x,F)$ and $\partial_m \bar V_2(x,F)$ vanish for $r\in\{2,3\}$ and $m\in\{1,3\}$.
Applying \citet[Theorem 3.2]{FisslerZiegel2016} yields the existence of continuously differentiable functions $h_{lm}\colon \interior(\A)\to\R$, $l,m\in\{1,2,3\}$, such that 
\(
\partial_m \bar S(x,F) = \sum_{i=1}^3 h_{mi}(x)\bar V_i(x,F)
\)
for $m\in\{1,2,3\}$.
Since we assume that $\bar S(\cdot, F)$ is twice continuously differentiable for any $F\in\F$, the second order partial derivatives need to commute. Let $t=T(F)$. Then $\partial_1\partial_2\bar S(t,F) = \partial_2 \partial_1 \bar S(t,F)$ is equivalent to
\(
h_{21}(t) f(t_1) = h_{12}(t) f(t_2).
\)
This needs to hold for all $F\in\F$. The variation in the densities implied by Assumption (V\ref{ass:V4}) in combination with the surjectivity of $T$ yield that $h_{12} \equiv h_{21} \equiv 0$ on $\interior(\A)$. Similarly, evaluating $\partial_1\partial_3 \bar S(x,F) = \partial_3\partial_1 \bar S(x,F)$ and $\partial_2\partial_3 \bar S(x,F) = \partial_3\partial_2 \bar S(x,F)$ at $x=t=T(F)$ yields
\(
h_{13}(t) = h_{31}(t)f(t_1), h_{23}(t) = h_{32}(t)f(t_2).
\) 
Using again Assumption (V\ref{ass:V4}) as well as the surjectivity of $T$, this implies that 
\(
h_{13} \equiv h_{31} \equiv h_{23} \equiv  h_{32} \equiv0.
\) 
So we are left with characterising $h_{mm}$ for $m\in\{1, 2,3\}$. Note that Assumption (V\ref{ass:V1}) implies that for any $x = (x_1,x_2,x_3)\in\interior(\A)$ there are two distributions $F_1,F_2\in\F$ such that $(F_1(x_1) - \a, F_1(x_2)  -\b)^\intercal$ and $(F_2(x_1) - \a, F_2(x_2)  -\b)^\intercal$ are linearly independent. Then, the requirement that 
\[
\partial_1\partial_2 \bar S(x,F) = \partial_1h_{22}(x)(F(x_2) - \b) = \partial_2 h_{11}(x)(F(x_1) - \a)  = \partial_2\partial_1\bar S(x,F)
\]
for all $x\in\interior(\A)$ and for all $F\in\F$ implies that $\partial_1h_{22} \equiv \partial_2 h_{11} \equiv 0$. 
Starting with $\partial_1\partial_3 \bar S(x,F) = \partial_3\partial_1 \bar S(x,F)$, implies that 
\(
\partial_1 h_{33} \bar V_3(x,F) = \big(\partial_3 h_{11}(x) + h_{33}(x)/(\b-\a)\big) \bar V_1(x,F).
\)
Again, Assumption (V\ref{ass:V1}) implies that there are $F_1,F_2\in\F$ such that $\big(\bar V_1(x,F_1), \bar V_3(x,F_1)\big)^\intercal$ and $\big(\bar V_1(x,F_2), \bar V_3(x,F_2)\big)^\intercal$ are linearly independent. Hence, we obtain that $\partial_1h_{33}\equiv0$ and $\partial_3 h_{11} \equiv - h_{33}/(\b-\a)$. With the same argumentation and starting from $\partial_2\partial_3 \bar S(x,F) = \partial_3\partial_2 \bar S(x,F)$ one can show that  $\partial_2h_{33}\equiv0$ and $\partial_3 h_{22} \equiv h_{33}/(\b-\a)$.
%
This means there exist functions 
$c_1 \colon \{(x_1,x_3)\in\R^2\,|\,\exists (z_1,z_2,z_3)\in \interior(\A), \ x_1=z_1, x_3 = z_3\}\to\R$, 
$c_2 \colon \{(x_2,x_3)\in\R^2\,|\,\exists (z_1,z_2,z_3)\in \interior(\A), \ x_2=z_2, x_3 = z_3\}\to\R$, 
$c_3 \colon 
\interior(\A)'_{3}\to\R$, 
and some $z\in \interior(\A)'_{3}$ such that for any $x = (x_1,x_2,x_3)\in\interior(\A)$ it holds that $h_{33}(x) = c_3(x_3)$,
\begin{align*}
h_{11}(x) &= c_{\min}(x_1,x_3) = - \frac{1}{\b - \a } \int_z^{x_3} c_3(z)\dint z + b_1(x_1),\\
h_{22}(x) &= c_{\max}(x_2,x_3) = \frac{1}{\b - \a} \int_z^{x_3} c_3(z) \dint z + b_2(x_2),
\end{align*}
where $b_r\colon \interior(\A)'_{r}\to\R$, $r\in\{1,2\}$. Due to the fact that any component of $T$ is mixture-continuous\footnote{For convex $\F$ a functional $T\colon \F\to\R^k$ is called mixture-continuous if for any $F,G\in\F$ the map $[0,1]\ni \lambda \mapsto T((1-\lambda)F + \lambda G)$ is continuous.} and since $\F$ is convex and $T$ surjective, the projection $ \interior(\A)'_{3}$ is an open interval. Hence, $[\min(z,x_3), \max(z,x_3)]\subset \interior(\A)'_{3}$.
Due to Assumptions (V\ref{ass:V3}) and (S\ref{ass:S2}), \citet[Theorem 3.2]{FisslerZiegel2016} implies that $c_1, c_2, c_3$ are locally Lipschitz continuous.

The above calculations imply that the Hessian of the expected score, $\nabla^2\bar S(x,F)$, at its minimiser $x=t = T(F)$,
is a diagonal matrix with entries  $c_1(t_1,t_3)f(t_1)$, $c_2(t_2,t_3)f(t_2)$, and $c_3(t_3)$.
As a second order condition $\nabla^2\bar S(t,F)$ must be 
positive semi-definite. Invoking the surjectivity of $T$ once again, this shows that $c_1,c_2, c_3\ge0$. More to the point, invoking the continuous differentiability of the expected score and the fact that $S$ is strictly $\F$-consistent for $T$ 
one obtains that 
for any $F\in\F$ with $t=T(F)$ and for any $v\in\R^3$, $v\neq 0$, there exists an $\eps>0$ such that
$\frac{\diff}{\diff s}\bar S(t+sv, F)$ is negative for all $s\in(-\eps,0)$, zero for $s=0$ and positive for all $s\in(\eps,0)$
For $v = e_3 = (0,0,1)^\intercal$, this means that for any $F\in\F$ with $t=T(F)$ there is an $\eps>0$ such that
$\frac{\diff}{\diff s}\bar S(t+se_3, F) = c_3(t_3+s)s  $
has the same sign as $s$ for all $s\in(-\eps,\eps)$.
Therefore, $c_3(t_3+s)>0$ for all $s\in(-\eps,\eps)\setminus\{0\}$. 
Using the surjectivity of $T$ and invoking a compactness argument, $c_3$ attains a 0 only finitely many times on any compact interval. 
Recall that $ \interior(\A)'_{3}$ is an open interval. Hence, it can be approximated by an increasing sequence of compact intervals. Therefore, $c_3^{-1}(\{0\})$ is at most countable and therefore a Lebesgue null set. With similar arguments one can show that for any $x_3\in  \interior(\A)'_{3}$, the sets $\{x_1\in\R\,|\,\exists (z_1,z_2,z_3)\in\interior(\A),\ x_1=z_1,\ x_3=z_3,\ c_1(x_1,x_3)=0\}$ and $\{x_2\in[x_3, \infty)\,|\, \exists (z_1,z_2,z_3)\in\interior(\A),\ x_2=z_2,\ x_3=z_3,\ c_2(x_2,x_3)=0\}$ are at most countable and therefore also Lebesgue null sets.

Finally, using Proposition 1 in \cite{Erratum} (recognising that $V$ is locally bounded) one obtains that $S$ is almost everywhere of the form \eqref{eq:S}. Moreover, it holds almost everywhere that $\phi''= c_3$ and $g_m' = b_m$ for $m\in\{1,2\}$. Hence, $\phi$ is strictly convex and the functions at \eqref{eq:G1} and \eqref{eq:G2} are strictly increasing.
\end{proof}

Combining Theorems \ref{thm:sufficiency} and \ref{thm:necessity}, one can show that the scoring functions given at \eqref{eq:S} are essentially the only strictly consistent scoring functions for the triplet $(\VaR_{\a}, \VaR_{\b}, \RVaR_{\a, \b})$ on the action domain $\A= \{(x_1,x_2,x_3)\in\R^3\,|\,c_{\min} \le  x_1\le x_3\le x_2 \le  c_{\max}\}$.

\begin{cor}\label{cor:char}
Let $\A= \{(x_1,x_2,x_3)\in\R^3\,|\,c_{\min} \le x_1\le x_3\le x_2 \le c_{\max}\}$ for some $-\infty\le c_{\min}<c_{\max}\le\infty$. Under the conditions of Theorem \ref{thm:necessity}, a scoring function $S\colon\A\times\R\to\R$ is strictly $\F$-consistent for $T = (\VaR_{\a}, \VaR_{\b}, \RVaR_{\a, \b})$, $0<\a<\b<1$, if and only if it is of the form \eqref{eq:S} almost everywhere
satisfying conditions \eqref{item:conv}, \eqref{item:incr}, \eqref{item:integr}.
Moreover, the function $\phi'\colon[c_{\min},c_{\max}]\to\R$ is necessarily bounded.
\end{cor}

\begin{proof} For the proof it suffices to show that for $r \in \{1,2\}$, $G_{r,x_3}$ defined in \eqref{eq:G1}, \eqref{eq:G2} is not only increasing on $\A_{r,x_3}'$ for any $x_3 \in \A_3'$ but on $\A_r' = [c_{\min},c_{\max}]$. For $x_3\in [c_{\min},c_{\max}] = \A_3'$, we have $\A_{1,x_3}' = [c_{\min},x_3]$ and $\A_{2,x_3}' = [x_3,c_{\max}]$. Let $x_3\in\A'_3$ and $x_1, x_1'\in\A'_1$ with $x_1<x_1'$. If $x_1, x_1'\in\A'_{1,x_3}$ there is nothing to show. If however $x_3<x_1'$, then $x_1,x'_1\in\A'_{1,x'_1}$. This means that 
\begin{align*}
0&\le g_1(x'_1) - g_1(x_1) - (x'_1 - x_1) \phi'(x_1')/(\b-\a)\\
&\le g_1(x'_1) - g_1(x_1) - (x'_1 - x_1) \phi'(x_3)/(\b-\a)\,
\end{align*}
where the second inequality stems from the fact that $\phi'$ is increasing.
If the function $G_{1,x'_1}$ is strictly increasing, then the first inequality is strict. The argument for $G_{2,x_3}$ works analogously.
\end{proof}

\begin{rem}\label{rem:difference of minima}
Note the structural difference of Theorems \ref{thm:sufficiency} and \ref{thm:necessity} to \citet[Theorem 1]{FrongilloKash2020}, \citet[Proposition 4.14]{Brehmer2017} and in particular \citet[Theorem 5.2 and Corollary 5.5]{FisslerZiegel2016}. Our functional of interest, $\RVaR_{\a,\b}$ with $0<\a<\b<1$, is not a minimum of an expected scoring function\,---\,or \emph{Bayes risk}\,---, but a difference of minima of two scoring functions. Indeed, while $\ES_\b(F) = -\frac{1}{\b}\bar S_\b(\VaR_\b(F),F)$, we have that
\[
\RVaR_{\a,\b}(F) = -\frac{1}{\b-\a}\big(\bar S_\b(\VaR_\b(F),F) - \bar S_\a(\VaR_\a(F),F)\big)\,.
\]
This structural difference is reflected in the minus sign appearing at \eqref{eq:G1}. In particular, it means that the functions $g_1$ and $g_2$ cannot identically vanish if we want to ensure strict consistency of $S$, whereas the corresponding functions in Theorem 5.2 in \cite{FisslerZiegel2016} may well be set to zero.
\citet[Theorem 2]{FrongilloKash2020} generalises our results and presents an elicitability result of any linear combination of Bayes risks.
\end{rem}

Concrete examples for choices of the functions $g_1$, $g_2$, and $\phi$ for the scoring function $S$ at \eqref{eq:S} are given and discussed in Section \ref{sec:sim}.

\section{Translation invariance and homogeneity}\label{sec:secondarycrit}

There are many choices for the functions $g_1$, $g_2$, and $\phi$ appearing in the formula for the scoring function $S$ at \eqref{eq:S}. Often, these choices can be limited by imposing secondary desirable criteria on $S$. In this section we show that, unfortunately, standard criteria (\cite{Patton2011, NoldeZiegel2017, FisslerZiegel2019}) such as translation invariance and positive homogeneity are not fruitful for RVaR.

If one is interested in scoring functions with an action domain of the form $\A= \{x\in\R^3\,|\,c_{\min} \le  x_1\le x_3\le x_2 \le c_{\max}\}$ possessing the additional property of translation invariant score differences, the only sensible choice is $c_{\min}=-\infty$, $c_{\max} = \infty$, amounting to the maximal action domain $\A_0$.
Similarly, for scoring functions with positively homogeneous score differences, the most interesting choices for action domains are $\A = \A_0$, $\A = \A_0^+ = \{(x_1,x_2,x_3) \in \R^3\,|\,0\le x_1\le x_3 \le x_2\}$ or $\A = \A_0^- = \{(x_1,x_2,x_3) \in \R^3\,|\,x_1\le x_3 \le x_2\le 0\}$.

\begin{prop}[Translation invariance]\label{prop: Translation invariance}
Under the conditions of Theorem \ref{thm:necessity} there are no strictly $\F$-consistent scoring functions for $(\VaR_\a,\VaR_\b, \RVaR_{\a,\b})$ on $\A_0$ with translation invariant score differences.
\end{prop}

\begin{proof}
Using Theorem \ref{thm:necessity} any strictly $\F$-consistent scoring function for $T$ must be of the form \eqref{eq:S} where in particular $\phi$ is strictly convex, twice differentiable, and $\phi'$ is bounded. Assume that $S$ has translation invariant score differences. That means that the function $\Psi\colon\R\times\A_0\times\A_0\times\R\to\R$,
\begin{align*}
\Psi(z,x,x',y) &= S(x_1 + z, x_2 + z,x_3+z,y+z) - S(x'_1+z,x'_2+z,x'_3+z,y+z) \\
&- S(x_1,x_2,x_3,y) + S(x'_1,x'_2,x'_3,y)
\end{align*}
vanishes. Then, for all $x\in\A_0$ and for all $z,y\in\R$
\[
0= \frac{\diff}{\diff x_3} \Psi(z,x,x',y) = \big(\phi''(x_3 + z) - \phi''(x_3)\big)\Big(x_3 + \frac{1}{\b-\a}\big(S_\b(x_2,y) - S_\a(x_1,y)\big)\Big)\,.
\]
Therefore, $\phi''$ needs to be constant. Since $\phi$ is convex and that means that $\phi'(x_3) =d x_3 + d'$ with $d>0$. But since $\A'_3 = \R$, $\phi'$ is unbounded, which is a contradiction.
\end{proof}
The proof of Proposition \ref{prop: Translation invariance} closely follows the one of Proposition 4.10 in \cite{FisslerZiegel2019}. The fact that the latter assertion entails a positive result has the following background: 
The strictly consistent scoring function for $(\VaR_\a, \ES_\a)$ given in \citet[Proposition 4.10]{FisslerZiegel2019} works only on a very restricted action domain. To guarantee strict consistency on such an action domain, one would need a refinement of Theorem \ref{thm:sufficiency} in the spirit of \citet[Proposition 2]{Erratum}. However, since such a positive result on a quite restricted action domain is practically irrelevant, we dispense with such a refinement and only state the relevant negative result here.

\begin{prop}[Homogeneity]\label{prop: Homogeneity}
Under the conditions of Theorem \ref{thm:necessity} there are no strictly $\F$-consistent scoring functions for $(\VaR_\a,\VaR_\b, \RVaR_{\a,\b})$ on $\A\in\{\A_0, \A_0^+, \A_0^-\}$ with positively homogeneous score differences.
\end{prop}

\begin{proof}
Using Theorem \ref{thm:necessity} any strictly $\F$-consistent scoring function for $T$ must be of the form \eqref{eq:S} where in particular $\phi$ is strictly convex, twice differentiable, and $\phi'$ is bounded. Assume that $S$ has positively homogeneous score differences of some degree $b\in\R$. That means that the function $\Psi\colon (0,\infty)\times\A\times\A\times\R\to\R$,
\begin{align*}
\Psi(c,x,x',y) = S(cx,cy) - S(cx',cy) - c^bS(x,y) + c^bS(x',y) 
\end{align*}
vanishes. Therefore, for all $x\in\A$, for all $y\in\R$ and all $c>0$
\be{eq:hom}
0= \frac{\diff}{\diff x_3} \Psi(z,x,x',y) = \big(c^2\phi''(cx_3) - c^b\phi''(x_3)\big)\Big(x_3 + \frac{1}{\b-\a}\big(S_\b(x_2,y) - S_\a(x_1,y)\big)\Big)\,.
\ee
For the sake of brevity, we only consider the case $\A = \A_0^-$, the other cases being similar.
Equation \eqref{eq:hom} implies that that $\phi''(-x_3) = \phi''(-1)x_3^{b-2}$ for any $x_3 >0$. 
 Due to the strict convexity of $\phi$, we need that $\phi''(-1)>0$. However, for $b\ge 1$, $\inf_{x_3>0}\phi'(-x_3) = -\infty$ and for $b\le 1$, $\sup_{x_3>0}\phi'(-x_3) = \infty$. Hence, $\phi'$ cannot be bounded.
\end{proof}

\begin{rem}
The negative result of Proposition \ref{prop: Homogeneity} should be compared with the results of Theorem C.3 in \cite{NoldeZiegel2017} characterising homogeneous strictly consistent scoring functions for the pair $(\VaR_\b, \ES_\b)$. Since they use a different sign convention for $\VaR$ and $\ES$ than we do in this paper, their choice of the action domain $\R\times(0,\infty)$ corresponds to our choice $\A_0^-$. When interpreting $\RVaR_{\a,\b}$ as a risk measure, negative values of $\RVaR$ are the more interesting and relevant ones, using our sign convention. 
Inspecting the proof of Proposition \ref{prop: Homogeneity} and of  Proposition \ref{prop:simplecond}(i) one makes the following observation: For $b\ge1$, \cite{NoldeZiegel2017} state an impossibility result for their choice of action domain. In fact, the problem occurring in our context is that $\phi'$ is not bounded \textit{from below}. In Proposition \ref{prop:simplecond} this property is implied by the fact that the function $G_{2,x_3}$ at \eqref{eq:G2} is increasing. And it is exactly such a condition that is also present for strictly consistent scoring functions for the pair $(\VaR_\b, \ES_\b)$; see Theorem 5.2 in \cite{FisslerZiegel2016}. On the other hand, the complication for $b<1$ stems from the fact that $\phi'$ is not bounded \textit{from above}. This condition is related to the monotonicity of $G_{1,x_3}$ at \eqref{eq:G1}. Such a condition is not present for strictly consistent scoring functions for the pair $(\VaR_\b, \ES_\b)$. Correspondingly, there can be homogeneous and strictly consistent scoring functions for $b<1$ for this pair \citep{NoldeZiegel2017} while this is not possible for the triplet $(\VaR_\a,\VaR_\b, \RVaR_{\a,\b})$.
\end{rem}

\section{Mixture representation of scoring functions}\label{sec:mixture}

When forecasts are compared and ranked with respect to consistent scoring functions, one has to be aware that in the presence of non-nested information sets, model mis-specification and/or finite samples, the ranking may depend on the chosen consistent scoring function \citep{Patton2020}. In the specific case of $(\VaR_\a, \VaR_\b, \RVaR_{\a,\b})$, the forecast ranking may depend on the specific choice for the functions $g_1$, $g_2$, and $\phi$ appearing in Theorem \ref{thm:sufficiency}. A possible remedy to this problem is to compare forecasts simultaneously with respect to \emph{all} consistent scoring functions in terms of Murphy diagrams as introduced by \citet{EhmETAL2016}. Murphy diagrams are based on the fact that the class of all consistent scoring functions can be characterised as a class of mixtures of elementary scoring functions that depend on a low-dimensional parameter. The following theorem provides such a mixture representation for the scoring functions at \eqref{eq:S}. 
The applicability is illustrated in Section \ref{sec:sim}.
Recall that $S_\alpha(x,y) = (\one\{y \le x\}-\alpha)x - \one\{y \le x\}y$.

\begin{thm}\label{thm:mixture}
Let $0 < \alpha < \beta < 1$. 
Any scoring function $S:[c_{\min}, c_{\max}]^3 \times \R \to \R$ 
of the form \eqref{eq:S} with $a\colon \R \to \R$ chosen such that $S(y,y,y,y) = 0$ can be written as
\begin{equation}\label{eq:mixture}
S(x_1,x_2,x_3,y) = \int L^1_v(x_1,y) \dint H_1(v) + \int L^2_v(x_2,y) \dint H_2(v) + \int L^3_v(x_1,x_2,x_3,y)\dint H_3(v),
\end{equation}
where
\begin{align*}
L^1_v(x_1,y) &= (\one\{y \le x_1\} - \alpha)(\one\{v \le x_1\} - \one\{v \le y\})\\
L^2_v(x_2,y) &= (\one\{y \le x_2\} - \beta)(\one\{v \le x_2\} - \one\{v \le y\})\\
L^3_v(x_1,x_2,x_3,y) & = \frac{1}{\beta-\alpha}\Big(\one\{v > x_3\}(S_\alpha(x_1,y)+\alpha y) + \one\{v \le x_3\}(S_\beta(x_2,y)+\beta y)\Big)\\
&\quad+ (\one\{v \le x_3\}-\one\{v \le y\})v,
\end{align*}
and $H_1$, $H_2$ are locally finite measures on $[c_{\min}, c_{\max}]$ 
and $H_3$ is a finite measure on $[c_{\min}, c_{\max}]$. 
If $H_3$ puts positive mass on all open intervals, then $S$ is strictly consistent.
Conversely, for any choice of measures $H_1, H_2, H_3$ with the above restrictions, we obtain a scoring function of the form \eqref{eq:S}.
\end{thm}
\begin{proof}
An increasing function $h:[c_{\min}, c_{\max}] \to \R$ can always be written as
\begin{equation}\label{eq:incmeasure}
h(x) = \int(\one\{v \le x\} - \one\{v \le z\})\dint H(v) + C, \quad x \in [c_{\min}, c_{\max}],
\end{equation}
for some locally finite measure $H$, and some $z \in [c_{\min}, c_{\max}]$, $C \in \R$. 
The function $h$ is strictly increasing if and only if $H$ is strictly positive, i.e., it puts positive mass on all open non-empty intervals. Furthermore, the one-sided derivatives of $h$ are bounded below by $\lambda>0$ if and only if $H(A) \ge \lambda\mathcal{L}(A)$ for all Borel sets $A\subseteq  [c_{\min}, c_{\max}]$, where $\mathcal{L}$ is the Lebesgue measure on $\R$. 

Using the arguments from Proposition \ref{prop:simplecond}, it is no loss of generality to show the assertion for a score $S$ such that $\lambda(\b - \a) = -\inf_x \phi'(x) = \sup_x \phi'(x)$ and the one-sided derivatives of $g_1$, $g_2$ are bounded from below by $\lambda>0$.

Then, there is a measure $H_3$ on $[c_{\min}, c_{\max}]$ such that $H_3([c_{\min}, c_{\max}]) = 2\lambda (\b-\a)$, which is strictly positive if and only if $\phi$ is strictly convex, such that for all 
for all $x_3 \in [c_{\min},c_{\max}]$, we have
\begin{align*}
\phi'(x_3) &= \int\one\{v \le x_3\}\dint H_3(v) - \lambda(\beta-\alpha) = \int\left(\one\{v \le x_3\} - \frac{1}{2}\right)\dint H_3(v).
\end{align*}
Using Fubini's theorem, we find that
\begin{align*}
&\phi(x_3) - \phi(y) = \int (\one\{w \le x_3\}-\one\{w \le y\})\phi'(w)\, \diff w \\
&= \int (\one\{w \le x_3\}-\one\{w \le y\})\int\left(\one\{v \le w\} - \frac{1}{2}\right)\dint H_3(v)\, \diff w\\
&= \int \int (\one\{w \le x_3\}-\one\{w \le y\})\one\{v \le w\}\diff w \, \diff H_3(v) - \int\frac{1}{2}(x_3 - y)\dint H_3(v)\\
&= \int \one\{v \le x_3\}(x_3 - v) - \one\{v \le y\}(y-v) - \frac{1}{2}(x_3 - y)\dint H_3(v).
\end{align*}
Using \eqref{eq:S}, \eqref{eq:incmeasure} and Proposition \ref{prop:simplecond} it is straight forward to check that a scoring function of the form \eqref{eq:S} can be written as in \eqref{eq:mixture} with $L_v^3$ replaced by 
\begin{align*}
\tilde L^3_v(x_1,x_2,x_3,y) & = \left(\one\{v \le x_3\}-\frac{1}{2}\right)\left(x_3 + \frac{1}{\beta-\alpha}(S_\beta(x_2,y)-S_\alpha(x_1,y))\right)\\
&\quad - \frac{1}{2}|x_3 - v| + \frac{1}{2}|y-v|,
\end{align*}
and locally finite measures $\tilde H_1$, $\tilde H_2$ on $[c_{\min},c_{\max}]$ instead of $H_1$, $H_2$ such that $\tilde H_i(A) \ge \lambda\mathcal{L}(A)$ for $i=1,2$, and for all Borel sets $A \subseteq \R$, and the measure measure $H_3$. 
We can write $\tilde H_i = H_i + \lambda \mathcal{L}$, $i=1,2$, for some locally finite measures $H_i$, $i=1,2$. Integrating $v \mapsto L_v^1$ with respect to $\lambda \mathcal{L}$, we obtain the function $\lambda(S_\alpha(x_1,y) + \alpha y)$, and analogously for $L_v^2$. Using that $H_3([c_{\min},c_{\max}]) = 2\lambda(\beta-\alpha)$ yields the claim with 
\begin{align*}
L^3_v(x_1,x_2,x_3,y) & = \frac{1}{2(\beta-\alpha)}(S_\beta(x_2,y) + \beta y + S_\alpha(x_1,y) + \alpha y)\\
&\quad + \left(\one\{v \le x_3\}-\frac{1}{2}\right)\left(x_3 + \frac{1}{\beta-\alpha}(S_\beta(x_2,y)-S_\alpha(x_1,y))\right)\\
&\quad - \frac{1}{2}|x_3 - v| + \frac{1}{2}|y-v|,
\end{align*}
which is equal to the formula given in the statement of the theorem. The scoring functions $L_v^1$ and $L_v^2$ are consistent for VaR at level $\alpha$ and $\beta$, respectively. 
The scoring function $L_v^3$ is of the form \eqref{eq:S} with $g_1(x) = g_2(x) = x/(2\b - 2\a)$ and $\phi(x) = |x-v|/2$, which renders it a consistent scoring function for $(\VaR_\a, \VaR_\b, \RVaR_{\a,\b})$.
The converse statement follows by direct computations.
\end{proof}

\section{Simulations}\label{sec:sim}

\begin{table}[t]
\begin{center}
\begin{tabular}{|c||c|}
\hline
Scoring function & $\phi'(x_3)$\\\hline\hline
$S_1$ &  $(\beta-\alpha)\tanh((\beta-\alpha)x_3)$\\\hline
$S_2$ &  $(\beta-\alpha)(2/\pi)\arctan((\beta-\alpha)x_3)$\\\hline
$S_3$ &  $(\beta-\alpha)(2\Phi((\beta-\alpha)x_3)-1)$\\\hline
$S_4$ &  $(\beta-\alpha)(-\one\{x_3 < c_1\} + \one\{x_3 > c_2\}$\\
& $+ \one\{c_1 \le x_3 \le c_2\}2(x_3-(c_1+c_2)/2)/(c_2 - c_1))$\\\hline
\end{tabular}
\caption{Examples of scoring functions. In all cases we choose $g_1(x_1) = x_1$ and $g_2(x_2) = x_2$. The parameters $c_1, c_2 \in \R$ satisfy $c_1 < c_2$.
\label{tab:1}}
\end{center}
\end{table}

This simulation study illustrates the usage of consistent scoring functions for the triplet $(\VaR_\a, \VaR_\b, \RVaR_{\a,\b})$ when comparing the predictive performances of different forecasts for this triplet, e.g., in the context of comparative backtests \citep{NoldeZiegel2017}.
Due to the negative results in Section \ref{sec:secondarycrit} it is challenging to suggest concrete examples for the choices of the functions $\phi$, $g_1$ and $g_2$ in \eqref{eq:S}. In Table \ref{tab:1}, we give some first suggestions. The scoring function $S_4$ is in the spirit of the Huber loss \cite[p.\ 79]{Huber1964}. It is only strictly consistent on $[c_1,c_2]^3$,
but remains consistent for all of $\R^3$.
We illustrate the discrimination ability of the suggested scoring functions with a slightly extended version of a simulation example of \citet{GneitingETAL2007} which has also been considered in \citet{FisslerETAL2016}.

\begin{figure}[t]
\includegraphics[width=0.32\textwidth]{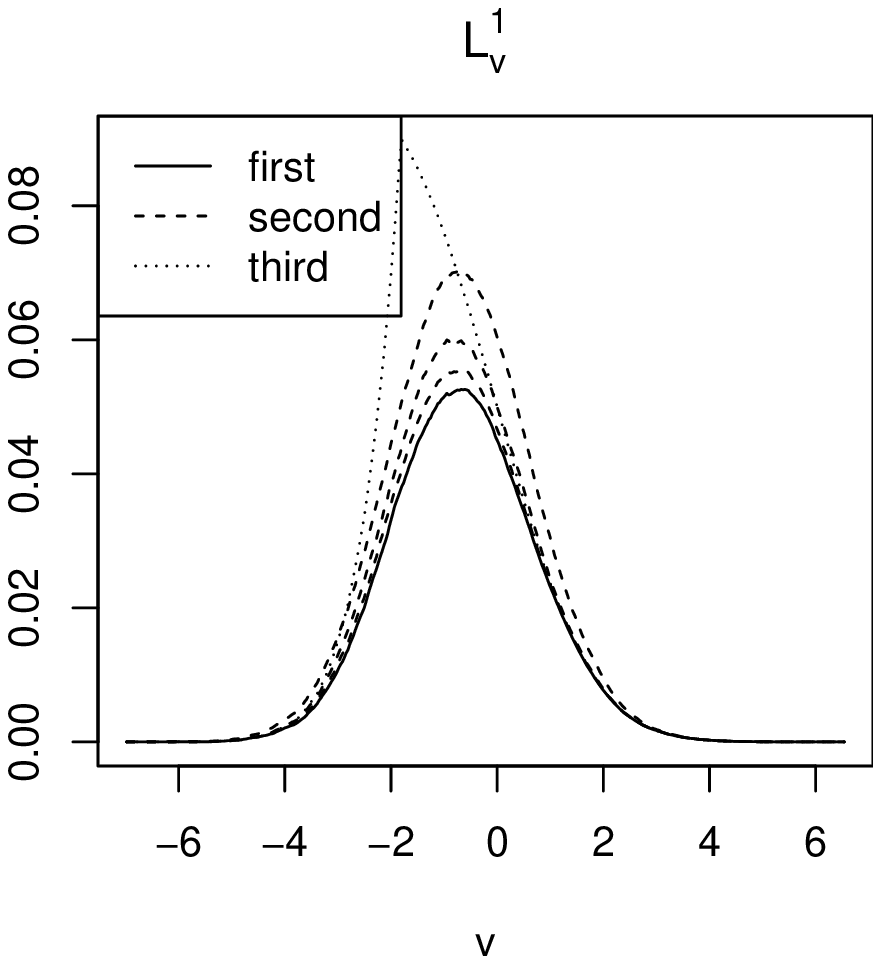}
\includegraphics[width=0.32\textwidth]{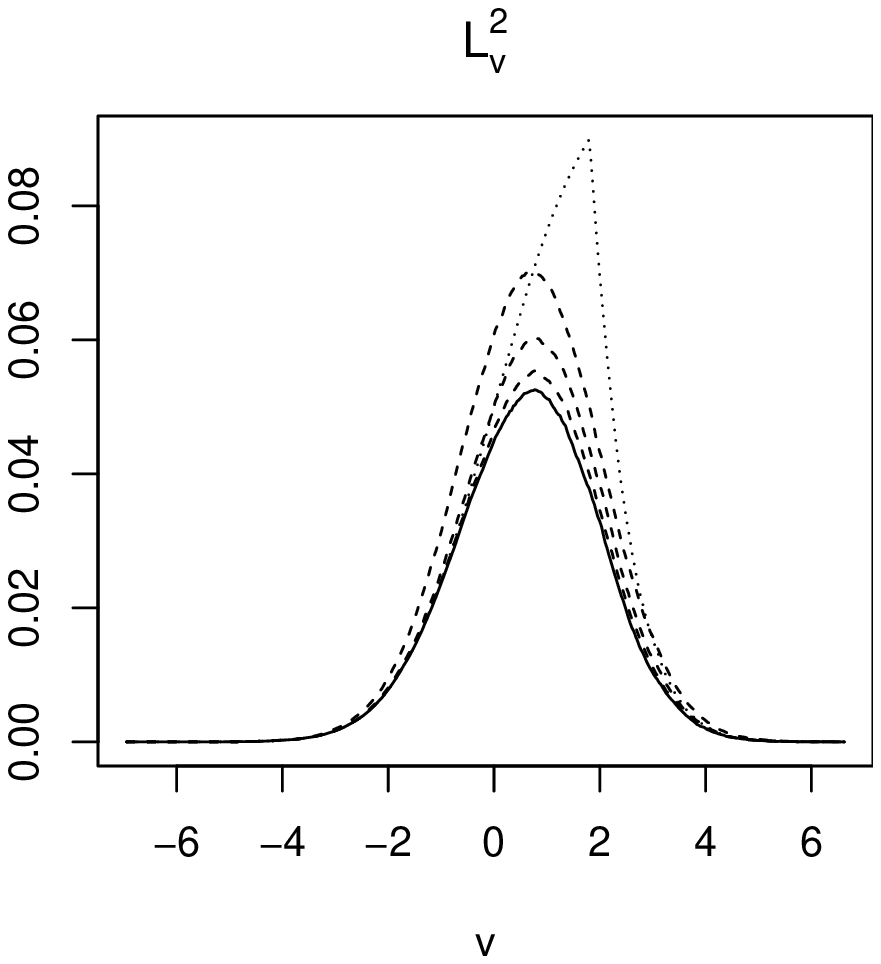}
\includegraphics[width=0.32\textwidth]{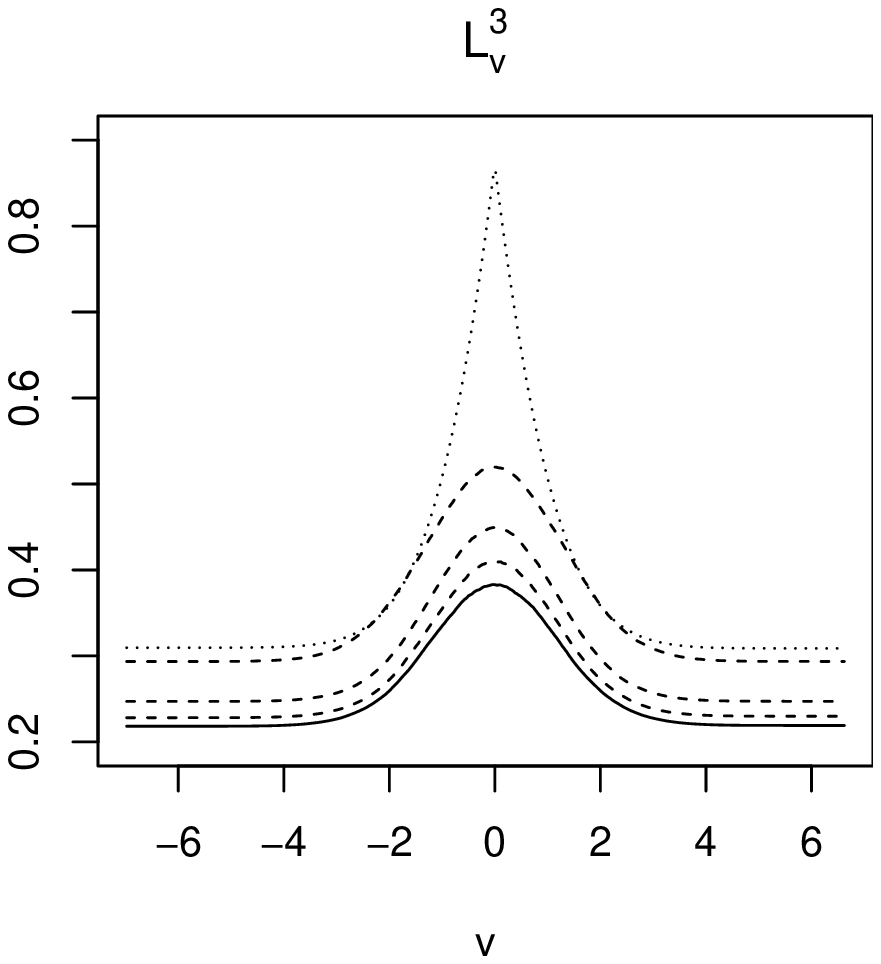}
\caption{Murphy diagrams for $\alpha=1-\beta=0.1$. Plots of 
expected elementary scores $L_v^1$, $L_v^2$, $L_v^3$ in terms of $v$ for the three forecasters described in the text. For the second forecaster, the curves correspond to $\sigma = 0.3,0.5,0.8$ from bottom to top.\label{fig:1}}
\end{figure}

\begin{figure}[t]
\includegraphics[width=0.32\textwidth]{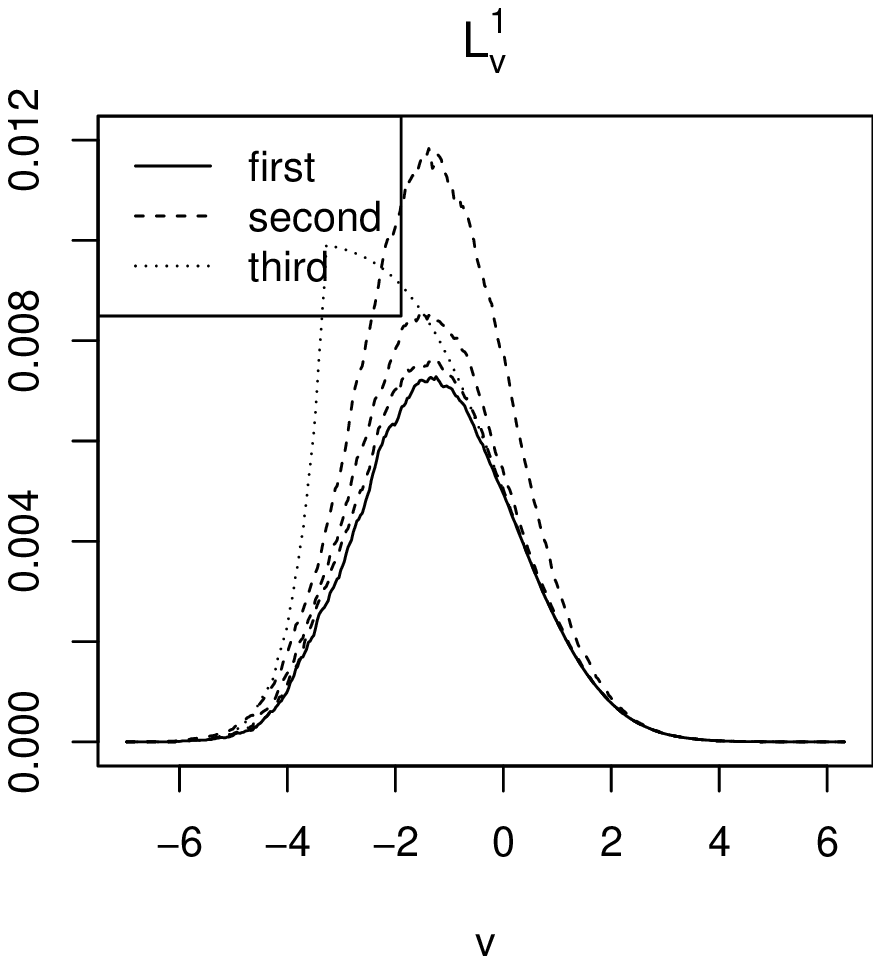}
\includegraphics[width=0.32\textwidth]{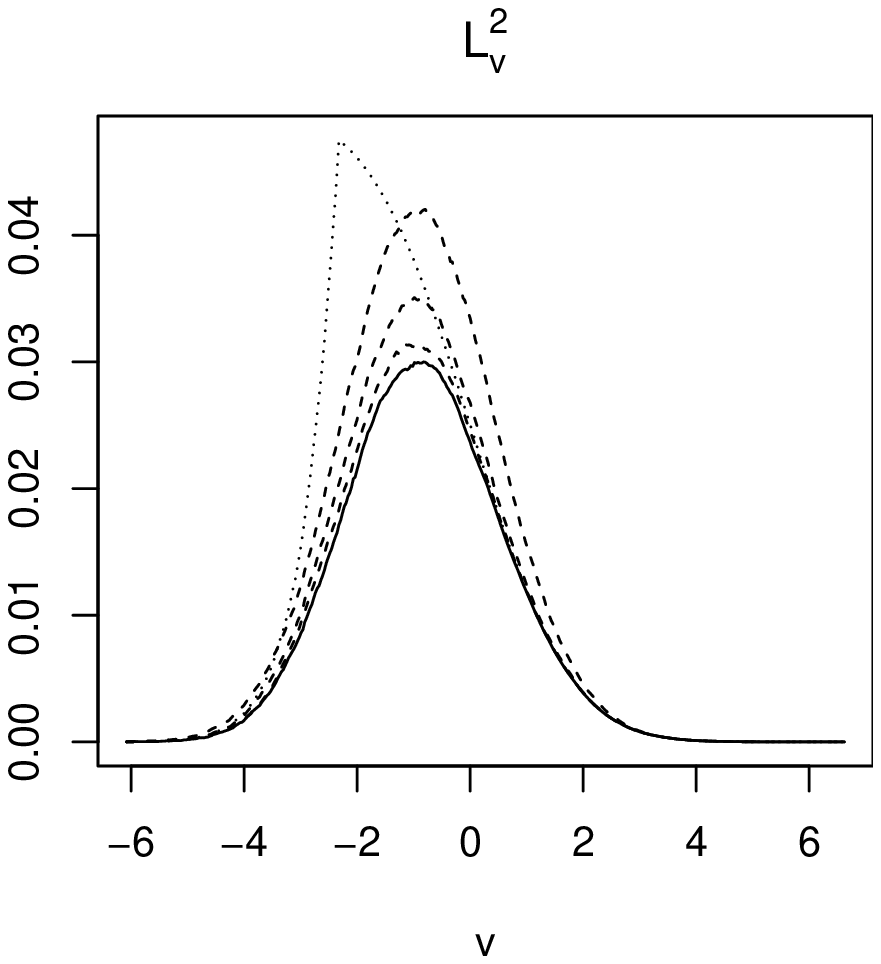}
\includegraphics[width=0.32\textwidth]{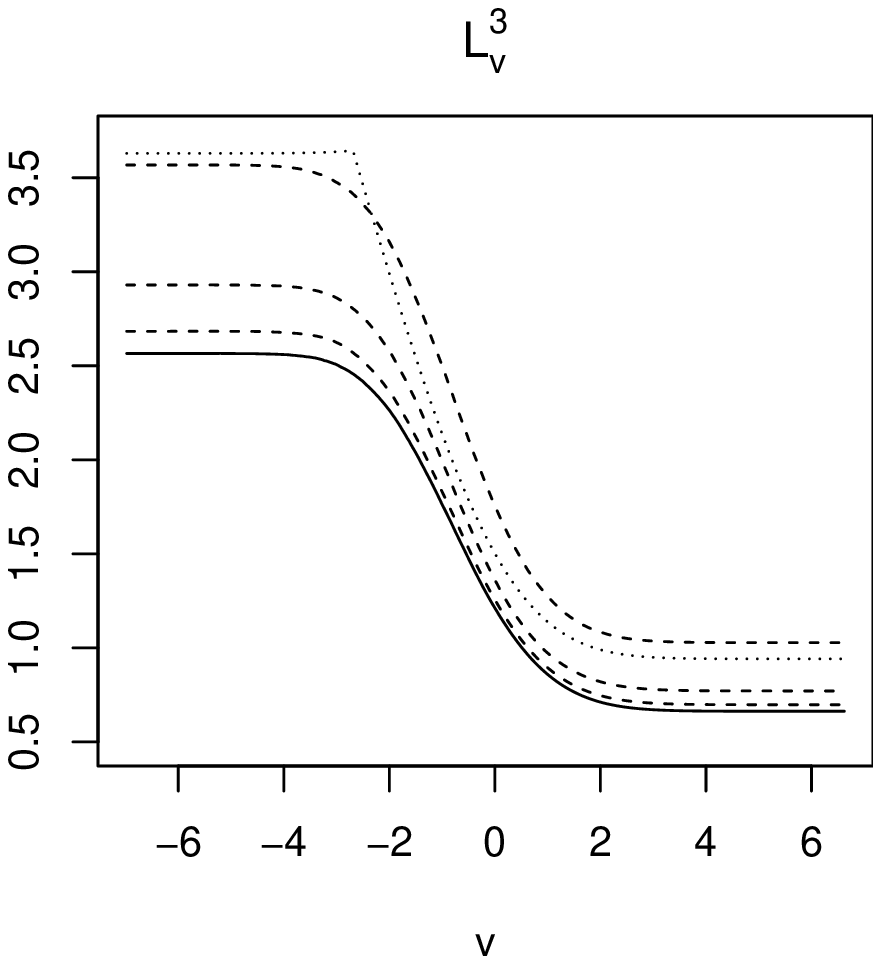}
\caption{Murphy diagrams for $\alpha=0.01$, $\beta=0.05$. Plots of expected elementary scores $L_v^1$, $L_v^2$, $L_v^3$ in terms of $v$ for the three forecasters described in the text. For the second forecaster, the curves correspond to $\sigma = 0.3,0.5,0.8$ from bottom to top.\label{fig:2}}
\end{figure}

We consider a data generating process $(Y_t)_{t=1,\dots,N}$ given by $Y_t = \mu_t + u_t$, where $(\mu_t)_{t=1,\dots,N}$ and $(u_t)_{t=1,\dots,N}$ are mutually independent sequences of i.i.d.\ standard normal random variables.
Suppose we have three different forecasters who provide point forecasts, aiming at correctly specifying $(\VaR_\a, \VaR_\b, \RVaR_{\a,\b})$ of the (conditional) distribution of $Y_t$.
The first forecaster has access to $\mu_t$ and uses the correct conditional distribution for prediction, that is, they predict 
\[
f_t = (f_{1,t},f_{2,t},f_{3,t}) = \left(\mu_t + \Phi^{-1}(\alpha), \mu_t + \Phi^{-1}(\beta), \mu_t - \frac{1}{\beta-\alpha}\left(\varphi(\Phi^{-1}(\beta))-\varphi(\Phi^{-1}(\alpha))\right)\right)
\]
for timepoint $t$, where $\varphi$ and $\Phi$ denote the density and quantile function of the standard normal distribution, respectively. The second forecaster predicts $g_t = (g_{1,t},g_{2,t},g_{3,t})$, where $g_{1,t}=f_{1,t}+\eps_t$, $g_{2,t}=f_{2,t}+\eps_t$ and $g_{3,t}=f_{3,t}+\eps_t$ and where $(\eps_t)_{t=1,\dots,N}$ is independent normally distributed noise with mean zero and variance $\sigma^2$. The third forecaster, $h_t = (h_{1,t},h_{2,t},h_{3,t})$, bases their predictions on the unconditional distribution of $Y_t$, that is $\mathcal{N}(0,2)$. Therefore, the forecasts take the form
\[
h_t = \left(\sqrt{2}\Phi^{-1}(\alpha),\sqrt{2}\Phi^{-1}(\beta),-\frac{\sqrt{2}}{\beta-\alpha}\left(\varphi(\Phi^{-1}(\beta))-\varphi(\Phi^{-1}(\alpha))\right)\right).
\]

It is clear that the first forecaster dominates the second and the third forecaster, that is, they will be preferred under any consistent scoring function. 
Indeed, invoking \cite{HolzmannEulert2014}, in case of the first and the second forecaster, the first one is ideal with respect to the information set $\sigma(\mu_t,\eps_t)$, whereas the second one is based on the same information set but is not ideal. In case of the first and the third forecaster, both forecasters are ideal but the information set of the first forecaster, $\sigma(\mu_t)$, is larger than the one of the third forecaster, which is the trivial $\sigma$-algebra. It will depend on the size of the variance $\sigma^2$ whether the second or the third forecaster is preferred. Figures \ref{fig:1} and \ref{fig:2} provide Murphy diagrams of all forecasters computed from a sample of size $N=100'000$, providing a good approximation of the population level. 
They are in line with our theoretical considerations above concerning the ranking of the three forecasts.

We compare the predictive performances using Diebold-Mariano tests \citep{DieboldMariano1995} based on the scoring functions in Table \ref{tab:1}. 
We consider samples of size $N=250$ and repeat our experiment 10'000 times. In the left panel of Table \ref{tab:2}, we consider the case that $\alpha = 1-\beta = 0.1$ where $\RVaR_{\a,\b}$ is a trimmed mean. We report the ratio of rejections of the null hypothesis that forecaster $i$ outperforms forecaster $j$, $i,j\in\{1,2,3\}$, $i\neq j$, evaluated in terms of the score $S$ at significance level $0.05$. E.g., for $i=1, j=2$, we consider the null hypothesis $\E[S(f_t,Y_t)]\le \E[S(g_t,Y_t)]$ for all $t=1, \ldots, N$, or in short, $f\preceq g$.
Analogously, in the right panel of Table \ref{tab:2}, we consider the case that $\alpha,\beta$ are both close to zero, that is, $\alpha=0.01$ and $\beta=0.05$, which is a setting that is relevant if $\RVaR_{\a,\b}$ is used as a risk measure. For the scoring function $S_4$, we have experimented a bit with the values $c_1$ and $c_2$ and report the results for the choices that worked best in our experiments. A systematic study on how to choose these two parameters goes beyond the scope of the present paper. 

\begin{table}[t]
\begin{center}
\begin{tabular}{|c||c|c|c|c|}
\hline
$H_0$ & $S_1$ & $S_2$ & $S_3$ & $S_4$\\\hline\hline
$f \preceq g$ & 0 & 0 & 0 & 0\\\hline
$g \preceq f$ & 0.304 & 0.406 & 0.417 & 0.624\\\hline
$f \preceq h$ & 0 & 0 & 0 & 0\\\hline
$h \preceq f$ & 1.000 & 1.000 & 1.000 & 1.000\\\hline
$g \preceq h$ & 0 & 0 & 0 & 0\\\hline
$h \preceq g$ & 0.999 & 0.998 & 0.992 & 0.998\\\hline
\end{tabular}
\qquad
\begin{tabular}{|c||c|c|c|c|}
\hline
$H_0$ & $S_1$ & $S_2$ & $S_3$ & $S_4$\\\hline\hline
$f \preceq g$ & 0 & 0 & 0 & 0.003\\\hline
$g \preceq f$ & 0.515 & 0.529 & 0.500 & 0.566\\\hline
$f \preceq h$ & 0 & 0 & 0 & 0\\\hline
$h \preceq f$ & 0.995 & 1.000 & 0.996 & 0.835\\\hline
$g \preceq h$ & 0.001 & 0 & 0 & 0\\\hline
$h \preceq g$ & 0.874 & 0.993 & 0.885 & 0.393\\\hline
\end{tabular}
\caption{Power of Diebold-Mariano tests at significance level $0.05$ for the scoring functions in Table \ref{tab:1} in the case that $\alpha = 1-\beta = 0.1$ (left panel), and $\alpha=0.01$, $\beta=0.05$ (right panel). In the first case we chose $-c_1 = c_2 = 12$ for the scoring function $S_4$, and $c_1 = -5$, $c_2 = 1$ in the second case. The null hypothesis $f \preceq g$ means that $\E [S(f_t,Y_t)] \le \E [S(g_t,Y_t)]$ for all $t=1, \ldots, N$ for the scoring function specified in the column label. We chose $\sigma^2 = 0.5^2$ for the forecaster $g$.\label{tab:2}}
\end{center}
\end{table}

For the situation of the left panel of Table \ref{tab:2} concerning $\a = 1-\b=0.1$, we can see that forecaster 1 (2) outperforms forecaster 3 with a power of 1 (almost 1) for all scoring functions used. For a comparison of forecaster 1 and forecaster 2, the situation is more interesting: While forecaster 1 outperforms forecaster 2 with regard to all scoring functions considered, the power of the tests (and the associated discrimination ability of the scoring functions) varies substantially. While $S_1$ leads to an empirical power of 0.304 for the null hypothesis $f \preceq g$, the score $S_4$ induces a power of 0.624 for the same null hypothesis.
The situation described in the right panel of Table \ref{tab:2} considering the parameter choice $\a=0.01$ and $\b=0.05$ leads to a different situation. The tests employing $S_1$, $S_2$ and $S_3$ have a similar power. In contrast, $S_4$ yields a considerably smaller power (0.393) for the null $h \preceq g$ than the other scores ($\ge 0.874$ for all cases). A more detailed study and comparison of other scoring functions and other situations is deferred to future work.

\section{Implications for regression}\label{sec:discussion}

After illustrating the usage of consistent scoring functions in forecast comparison and comparative backtesting in Section \ref{sec:sim}, 
we would like to outline how one can implement our results about the elicitability of the triplet $(\VaR_\a, \VaR_\b, \RVaR_{\a,\b})$, $0<\a<\b<1$ in a regression context. Then we would like to contrast our ansatz to other suggestions for regression of the $\a$-trimmed mean (which can be generalised to $\RVaR_{\a,\b}$). The most common alternative approaches in the literature on robust statistics are the \emph{trimmed least squares approach} and a two-step estimation procedure using the \emph{Huber skipped mean}.

\subsection{A joint regression framework for $(\VaR_\a, \VaR_\b, \RVaR_{\a,\b})$}\label{subsec:joint regression}

Let $(X_t,Y_t)_{t\in\mathbb N}$ be a time series with the usual notation that $Y_t$ denotes some real valued response variable and $X_t$ is a $d$-dimensional vector of regressors. Let $\Theta\subseteq \R^k$ be some parameter space and $M\colon \R^d\times \Theta \to \R^3$
a parametric model for $T = (\VaR_\a, \VaR_\b, \RVaR_{\a,\b})$, $0<\a<\b<1$.
We assume a correct model specification, that is, we assume that 
there is a unique $\theta_0\in\Theta$ such that 
\be{eq:model spec}
T(F_{Y_t|X_t}) = M(X_t, \theta_0) \quad \P\text{-a.s.} \ \text{for all } t\in\mathbb N,
\ee
where $F_{Y_t|X_t}$ denotes the conditional distribution of $Y_t$ given $X_t$. That means, $M(X_t, \theta_0)$ models jointly the conditional $\VaR_\a$, $\VaR_\b$ and the conditional $\RVaR_{\a,\b}$. Let $S$ be a strictly consistent scoring function of the form \eqref{eq:S} and suppose  the sequence $(X_t,Y_t)_{t\in\mathbb N}$ satisfies certain mixing conditions \cite[Corollary 3.48]{White2001} (in particular under independence). Then one obtains under additional moment conditions that, as $n\to\infty$,
\[
\frac{1}{n} \sum_{t=1}^n S(M(X_t,\theta), Y_t) - \frac{1}{n} \sum_{t=1}^n  \E\big[S(M(X_t,\theta), Y_t)\big] \to 0 \quad \P\text{-a.s.} 
\]
It is essentially this Law of Large Numbers result which allows for consistent parameter estimation with the empirical $M$-estimator $\widehat \theta_n = \argmin_{\theta_\in \Theta} n^{-1} \sum_{t=1}^n S(M(X_t,\theta), Y_t)$; see e.g.\ \cite{VanderVaart1998}, \cite{HuberRonchetti2009}, \cite{NoldeZiegel2017} and \cite{DFZ2020} for details. 

In summary, we can see that the complication of this procedure is that one needs to model the components $\VaR_\a$, $\VaR_\b$, even if one is only interested in $\RVaR_{\a,\b}$. The advantage is that one can substantially deviate from an i.i.d.\ assumption on the data generating process. One can deal with serially dependent, though mixing, and non-stationary data. One only needs the \emph{semiparametric stationarity} specified through \eqref{eq:model spec}.

\subsection{Trimmed least squares}

Most proposals for $M$-estimation and regression for $\RVaR_{\a,\b}$ in the field of robust statistics focus on the $\a$-trimmed mean, $\a\in(0,1/2)$, corresponding to $\RVaR_{\a,1-\a}$. But they can often be extended to the general case $0<\a<\b<1$ in a straightforward way. When this is the case, we describe the procedure in this more general manner. A majority of the proposals in the literature are commonly referred to as a \emph{trimmed least squares} (TLS) approach. However, strictly speaking, TLS actually subsumes different, though closely related estimation procedures.

 The first one was coined by \cite{KoenkerBasset1978}\,---\,cf.\ \cite{RuppertCarroll1980}\,---\,and constitutes a two-step $M$-estimator: In a first step, the $\a$- and $\b$-quantile are determined via usual $M$-estimation. Then, all values below the former and above the latter are omitted and $\RVaR_{\a,\b}$ is computed with an ordinary least squares approach. One can also express this procedure using order-statistics. Using the notation from Subsection \ref{subsec:joint regression}, an $M$-estimator for $\RVaR_{\a,\b}$ is given by
 \(
 \argmin_{z\in\R} \frac{1}{n} \sum_{i= [n\a]}^{[n\b]} (z - Y_{(i)})^2.
 \)
 Here, $Y_{(1)} \le \cdots \le Y_{(n)}$ is the order-statistics of the sample $Y_1, \ldots, Y_n$. While this procedure seems to work for a simplistic regression model (ignoring the regressors $X_t$ and only modelling the intercept part), it is not clear how to use it in a more interesting regression context, where one is actually interested in the \emph{conditional} distribution of $Y_t$ given $X_t$ rather than the unconditional distribution of $Y_t$. Moreover, since this approach uses the order statistics of the entire sample $Y_1, \ldots, Y_n$ to implicitly estimate the $\a$- and $\b$-quantile, it requires that these quantiles be constant in time. Hence, heteroscedasticity (in time) can lead to problems, even if $\RVaR_{\a,\b}$ is constant in time.
  
A second approach is described, for example, in \cite{Rousseeuw1984, Rousseeuw1985} and relies on order-statistics of the \emph{squared residuals}. It only seems to work for the $\a$-trimmed mean. To be more precise, and again using the notation from above, let $m\colon \R^d \times \Theta \to\R$ be a one-dimensional parametric model. Again, one assumes that there is a unique correctly specified model parameter $\theta_0\in\Theta$ such that 
\be{eq:model spec2}
\RVaR_{\a,1-\a}(F_{Y_t|X_t}) = m(X_t, \theta_0) \quad \P\text{-a.s.} \ \text{for all } t\in\mathbb N.
\ee
For each $\theta\in\Theta$, define the residuals $\eps_t(\theta) := Y_t - m(X_t,\theta)$ and the absolute residuals $r_t(\theta) := |\eps_t(\theta)|$. Define the order-statistics of the absolute residuals $0\le r_{(1)}(\theta) \le \cdots \le r_{(n)}(\theta)$ for a sample of size $n$. Then an $M$-estimator is defined via
\[
\widehat \theta_n = \argmin_{\theta\in\Theta} \frac{1}{n}\sum_{i=1}^{[n(1-2\a)]} r^2_{(i)}(\theta).
\]
While this procedure appears to be fairly similar to an ordinary least squares procedure with the respective computational advantages, one should recall that the trimming crucially depends on the choice of the parameter $\theta$. That means even if the model $m$ is linear in the parameter $\theta$, one generally yields a non-convex objective function with several local minima.
Interestingly, the trimming takes place only for residuals with large modulus. If the error distribution is symmetric, this procedure yields a consistent estimator for $\theta_0$ in an i.i.d.\ setting. If one wants to relax the assumption on the error distribution and is interested in modelling $\RVaR_{\a,\b}$ for general $0<\a<\b<1$ in \eqref{eq:model spec2}, one could come up with the following ad-hoc procedure: Consider the order-statistics of the \emph{residuals} $\eps_{(1)}(\theta) \le \cdots \le \eps_{(n)}(\theta)$. Then define an $M$-estimator via
\[
\widehat \theta_n =  \argmin_{\theta\in\Theta}\frac{1}{n} \sum_{i=[n\a]}^{[n\b]} |\eps_{(i)}(\theta)|^2\,.
\]
This procedure takes into account the asymmetric nature of trimming when dealing with $\b\neq 1-\a$ or $\b=1-\a$ and an asymmetric error distribution. 
However, as outlined above, this procedure can lead to problems in the presence of heteroscedasticity or general non-stationarity of the error distribution, if the conditional $\VaR_\a$ and $\VaR_\b$ of $Y_t$ given $X_t$ depends on $X_t$.
We would like to point out that, at the cost of additionally modelling the $\a$- and $\b$-quantile, the procedure using our strictly consistent scoring functions for the triplet $(\VaR_\a, \VaR_\b, \RVaR_{\a,\b})$ described in Subsection \ref{subsec:joint regression} does not rely on the usage of order-statistics and it can in general deal with heteroscedasticity. The only degree of `stationarity' is required through \eqref{eq:model spec}. Especially stationarity is deemed a too strong assumption in the context of financial data; see \cite{Davis2016}.

Finally, we would like to remark that there are further procedures belonging to the field of TLS. For instance, \cite{AtkinsonCheng1999} propose an adaptive procedure where the trimming parameter is data driven; see also \cite{CerioliETAL2018}. However, we see no apparent way how to use such procedures if one is interested in \emph{predefined} trimming parameters $\a$ and $\b$.

\subsection{Connections to Huber loss and Huber skipped mean}

In his seminal paper, \cite{Huber1964} introduced the famous \emph{Huber loss} $S(x,y) = \rho(x-y)$ where $\rho(t) = \frac12 t^2$ for $|t|\le k$ and $\rho(t) = k|t| - \frac12 k^2$ for $|t|>k$. Huber argues that the ``the corresponding [M-]estimator is related to Winsorizing'' \cite[p.\ 79]{Huber1964}. What obtained significantly less attention\,---\,maybe due to its lack of convexity\,---\,is another loss function he considers on the same page of the paper which is defined as $S(x,y) = \rho(x-y)$ for $\rho(t) = \frac12 t^2$ for $|t|\le k$ and $\rho(t) = \frac12 k^2$ for $|t|>k$. He writes about it: ``the corresponding [M-]estimator is a trimmed mean'' (\textit{ibidem}).

One could define an asymmetric version of the latter loss function by using $S_{k_1,k_2}(x,y) = \rho_{k_1,k_2}(x-y)$ with 
\[
\rho_{k_1,k_2}(t) = \begin{cases}
\frac12 k_1^2 & t<k_1\\
\frac12 t^2 & k_1\le t < k_2\\
\frac12 k_2^2 & t\ge k_2.
\end{cases}
\]
Assuming that $F$ is continuous with density $f$ for the sake of the simplicity of the argument,
the corresponding first-order condition for a minimum of the expected score $\bar S_{k_1,k_2}(x,F)$ is equivalent with
\begin{align*}
x = \frac{1}{F(k_2 - x) - F(k_1 - x)}\int_{k_1-x}^{k_2 - x} yf(y)\dint y.
\end{align*}
Now a suggestion similar to \citet[p.\ 876]{Rousseeuw1984} is to consider this loss with $k_1 = \VaR_\b(F)$ and $k_2 = \VaR_\a(F)$ stemming from some pre-estimate. However, one can see that the first order-condition is generally not solved by $\RVaR_{\a,\b}(F)$. Again, if one is interested in $M$-estimation for the trimmed mean or, more generally, RVaR, one should use the scoring functions introduced at \eqref{eq:S}.

\section*{Acknowledgements}
We would like to thank Timo Dimitriadis and Anthony C.\ Atkinson for insightful discussions about the topic, and Ruodu Wang, Rafael Frongillo, Tilmann Gneiting and Jana Hlavinov\'a for helpful suggestions which improved an earlier version of this paper. \\
Tobias Fissler is grateful to the Department of Mathematics at Imperial College London who funded his fellowship during which most of the work of this paper has been done. Johanna Ziegel is grateful for financial support from the Swiss National Science Foundation.

\appendix

\section*{Appendix}

We present a list of assumptions used in Section \ref{sec:results}. For more details about their interpretations and implications, please see \cite{FisslerZiegel2016} where they were originally introduced.

\begin{ass}[V]\label{ass:V1}
$\F$ is convex and for every $x\in\interior(\A)$ there are $F_1, \ldots, F_{k+1}\in\F$ such that 
\(
0\in \interior\left( \conv\left(\left\{ \bar V(x,F_1), \ldots, \bar V(x,F_{k+1})\right\}\right)\right)\,.
\)
\end{ass}
Note that if $V\colon\A\times \R\to\R^k$ is a strict $\F$-identification function for $T\colon\F\to\A$ which satisfies Assumption (V\ref{ass:V1}), then for each $x\in\interior(\A)$ there is an $F\in\F$ such that $T(F) = x$.
\setcounter{ass}{2}
\begin{ass}[V]\label{ass:V3}
The map $\bar V(\cdot, F)$ is continuously differentiable for every $F\in \F$.
\end{ass}

\begin{ass}[V]\label{ass:V4}
Let assumption (V\ref{ass:V3}) hold. For all $r\in\{1, \ldots, k\}$ and for all $t\in\interior(\A)\cap T(\F)$ there are $F_1, F_2 \in T^{-1}(\{t\})$ such that
\begin{align*}
\partial_l \bar V_l(t,F_1) = \partial_l \bar V_l(t,F_2)\quad \forall l\in\{1,\ldots, k\}\setminus \{r\}, && \partial_r \bar V_r(t,F_1) \neq \partial_r \bar V_r(t,F_2).
\end{align*}
\end{ass}

\setcounter{ass}{0}
\begin{ass}[F]\label{ass:F1}
For every $y\in\R$ there exists a sequence $(F_n)_{n \in \mathbb{N}}$ of distributions $F_n \in \F$ that converges weakly to the Dirac-measure $\delta_y$ such that the support of $F_n$ is contained in a compact set $K$ for all $n$. 
\end{ass}

\setcounter{ass}{0}
\begin{ass}[VS]\label{ass:VS1}
Suppose that the complement of the set
\[
C := \{(x,y) \in \A \times \R\;|\; \text{$V(x,\cdot)$ and $S(x,\cdot)$ are continuous at the point $y$}\}
\]
has $(k+d)$-dimensional Lebesgue measure zero.
\end{ass}

\setcounter{ass}{1}

\begin{ass}[S]\label{ass:S2}
For every $F\in \F$, the function $\bar S(\cdot, F)$ is continuously differentiable and the gradient is locally Lipschitz continuous. Furthermore, $\bar S(\cdot, F)$ is twice continuously differentiable at $t = T(F)\in \interior(\A)$. 
\end{ass}

\bibliographystyle{plainnat}

\end{document}